\documentclass[a4paper,12pt]{amsart}

\usepackage[english]{babel}

\usepackage[a4paper,left=30mm,right=30mm,top=35mm,bottom=35mm,marginpar=25mm]{geometry} 

\usepackage[T1]{fontenc}
\usepackage[utf8]{inputenc}
\usepackage{amsthm, amssymb, amsmath, dsfont, mathrsfs, bbold,amsfonts,}
\usepackage{mathtools}
\usepackage{graphicx}
\usepackage{stackrel}
\usepackage{enumerate}
\usepackage{xcolor}
\usepackage{enumitem}
\usepackage{esint}

\usepackage[colorlinks=true, citecolor=blue]{hyperref} 

\numberwithin{equation}{section}
\newtheorem{theorem}{Theorem}[section]
\newtheorem{lemma}[theorem]{Lemma}

\newtheorem{prop}[theorem]{Proposition}
\newtheorem{cor}[theorem]{Corollary}

\theoremstyle{definition}
\newtheorem{defin}[theorem]{Definition}
\newtheorem{rem}[theorem]{Remark}

\AtBeginDocument{
  \let\div\relax
  \DeclareMathOperator{\div}{div}
}

\DeclareMathOperator{\Capa}{cap}
\DeclareMathOperator{\dist}{dist}
\DeclareMathOperator{\diam}{diam}
\DeclareMathOperator{\graph}{graph}

\newcommand{\mres}{\mathbin{\vrule height 1.6ex depth 0pt width
0.13ex\vrule height 0.13ex depth 0pt width 1.3ex}}

\newcommand{\cpp}[2]{\int_{#1}{\vert\nabla #2\vert^p}dx}

\title[Quantitative isocapacitary inequality]{The sharp quantitative isocapacitary inequality (the case of $p$-capacity)}

\author[E. Mukoseeva]{Ekaterina Mukoseeva}
\address{E.M: Department of Mathematics and Statistics, P.O. Box 68 (Gustaf H\"allstr\"omin katu 2), FI-00014 University of Helsinki, Finland}
\email{ekaterina.mukoseeva@helsinki.fi}

\begin{document}

\begin{abstract}
We prove a sharp quantitative form of isocapacitary inequality
in the case of a general $p$. This work is a generalization
of the author's paper with Guido De Philippis and Michele Marini,
where we treated the case of $2$-capacity.
\end{abstract}

\maketitle

\section{Introduction}
	
\subsection{Background}
Let $\Omega\subset\mathds{R}^N$ be an open set. We define the \emph{$p$-capacity} of $\Omega$ as 
\begin{equation}\label{e:pcapdef}
\Capa_p(\Omega)=\inf_{u\in C_c^\infty(\mathds{R}^N)}\left\{\int_{\mathds{R}^N}{\vert\nabla u\vert^p}dx: u\ge1\text{ on }\Omega\right\}
\end{equation}
for $1<p<n$.
It  is easy to see that  for problem \eqref{e:pcapdef} there exists  a unique function\footnote{Here $D^{1,p}$ denotes homogeneous Sobolev space, i.e. $D^{1,p}=\overline{C^\infty_c(\mathds{R}^n)}^{\Vert\cdot\Vert_{D^{1,p}}}$, where $$\Vert u\Vert_{D^{1,p}}~:=~\Vert\nabla u\Vert_{L^p(\mathds{R}^n)}$$} $u\in D^{1,p}(\mathds{R}^n)$ called {\it capacitary potential} of $\Omega$ such that
\[
\int_{\mathds{R}^N} |\nabla u|^p=\Capa_p(\Omega) .
\]

Moreover, it satisfies the  Euler-Lagrange equation (if $\Omega$ is sufficiently smooth, e.g. Lipschitz):			
\begin{equation*}
\begin{cases}
\div(\vert\nabla u\vert^{p-2}\nabla u)=0\text{ in }\overline{\Omega}^c, \\
u=1\text{ on }\partial \Omega, \\
u(x)\rightarrow 0\text{ as }\vert x\vert\rightarrow\infty.
\end{cases}
\end{equation*}		
		
P\'olya-Szeg\"o principle yields the well-known {\it isocapacitary inequality}, telling that, among all sets with given volume, balls have the smallest possible $p$-capacity, namely
\begin{equation}\label{iso_p}
\Capa_p(\Omega)-\Capa_p(B_r)\ge 0.
\end{equation}
Here $r$ is such that $|B_r|=|\Omega|$, where $|\cdot|$ denotes the Lebesgue measure.

Indeed, to prove \eqref{iso_p} recall the notion of {\it Schwarz symmetrization}. Let $\Omega$ be an open set and let $u$ be its capacitary potential. Schwarz symmetrization provides us with a radially symmetric function $u^*$ such that, for every $t\in\mathds R$,
\begin{equation}\label{equim}
\left\vert\{x: u(x)>t\}\right\vert=\left\vert\{x: u^*(x)>t\}\right\vert.
\end{equation}
We use $u^*$ as a test function for the set $\{x\,:\, u^*(x)=1\}=B_r$  and we note that  \eqref{equim} yields that $|B_r|=|\Omega|$. Hence  
\[
\Capa_p(B_r)\leq\int_{\mathds{R}^N}{\vert\nabla u^*\vert^p}dx\le\int_{\mathds{R}^N}{\vert\nabla u\vert^p}dx
=\Capa_p(\Omega) \qquad |\Omega|=|B_r|,
\]
where the second inequality follows by  P\'olya-Szeg\"o principle. 

Inequality \eqref{iso_p} is rigid, that is, equality is attained only when $\Omega$ coincides with a ball, up to a set of zero $p$-capacity. To see that, one may use that the equality in P\'olya-Szeg\"o inequality is attained only when $u$ is almost everywhere equal to a translate of $u^*$ (see, for example, \cite{BZ}).

\medskip

It is natural to wonder whether this inequality is also stable, that is $\Omega\to B_r$, whenever $\Capa_p(\Omega)\to\Capa_p(B_r)$.
This indeed turns out to be true and the right choice of the distance
between sets is \emph{Fraenkel asymmetry}, defined below.

\begin{defin}
		Let $\Omega$ be an open set. The Fraenkel asymmetry of \(\Omega\),  $\mathcal{A}(\Omega)$, is defined as:
		\begin{equation*}
			\mathcal{A}(\Omega)=\inf\left\{\frac{\vert\Omega\Delta B\vert}{\vert B\vert}\,:\, B \text{ is a ball with the same volume as }\Omega\right\}.
		\end{equation*}
\end{defin}	

To our knowledge, the first result of this sort goes back to \cite{HHW}, where they considered the  case of  planar sets and $p=2$ 
\footnote{Note that for \(N=p=2\) the infimum  \eqref{e:pcapdef} is \(0\) and one has to use the notion of logarithmic capacity.} 
and of convex sets in general dimension.
The sharp stability inequality for $N=p=2$ has been given by Hansen and Nadirashvili  in \cite[Corollary 1]{HN}. 
For general dimension and $p$, the best result to our knowledge is due to Fusco, Maggi, and Pratelli in \cite{nonsharp} where they prove the following:

	\begin{theorem} [\cite{nonsharp}]\label{t:fmp}
		There exists a constant $c=c(N,p)$ such that, for any open set $\Omega$
		\begin{equation*}
			\frac{\Capa_p(\Omega)-\Capa_p(B_r)}{r^{N-p}}\geq c\,\mathcal{A}(\Omega)^{2+p}.
		\end{equation*}
	\end{theorem}

For the case $p=2$, the following inequality was obtained in 
a recent paper by De Philippis, Marini, and the author.

	\begin{theorem} [\cite{DPMM}]\label{thm:2-cap}
		There exists a constant $c=c(N,p)$ such that, for any open set $\Omega$
		\begin{equation*}
			\frac{\Capa_2(\Omega)-\Capa_2(B_r)}{r^{N-2}}\geq c\,\mathcal{A}(\Omega)^{2}.
		\end{equation*}
	\end{theorem}

The goal of this paper is to extend the result of Theorem \ref{thm:2-cap} to the case of a general $p$.

	\subsection{Main result and strategy of the proof}
	
	The following is the main result of the paper. By the scaling \(\Capa_p(\lambda \Omega)=\lambda^{N-p}\Capa_p(\Omega)\), we can also get the analogous result for  $\Omega$ with  arbitrary volume.

	\begin{theorem}\label{thm:main p-cap}
		Let $\Omega$ be an open set such that $\vert\Omega\vert=\vert B_1\vert$.
		Then there exists a constant $c=c(N,p)$ such that the following inequality holds:
				\begin{equation*}
					\Capa_p(\Omega)-\Capa_p(B_1)\geq c\,\mathcal{A}(\Omega)^2.
				\end{equation*}
	\end{theorem}
	
	Notice that by testing the inequality on ellipsoids, one can see that
the exponent $2$ on the right hand side of the inequality in Theorem
\ref{thm:main p-cap} cannot be improved.
	
	A natural way to tackle a problem like this is 
	quantifying the proof of inequality \eqref{iso_p}. 
	Indeed, that was done in \cite{nonsharp}.
	However, it seems that one is bound to get a non-sharp
	exponent on the right-hand side arguing in such a way.	
	So instead we employ the approach of \cite{CL}, 
	devised for proving sharp isoperimetric inequality.
	The method, called Selection Principle, can be described as follows.

	We first prove the theorem for smooth enough sets. This can be done by taking shape derivatives, writing Taylor expansion
	near the ball
	for $p$-capacity and providing relevant bounds for the second 
	derivative and the remainder term. That is done in Section \ref{fuglede}.	
	
	Then we prove the theorem for bounded sets of small asymmetry. 
	We argue by contradiction. Suppose there exists a sequence
	of sets $\Omega_j$ such that $\frac{\Capa_p(\Omega_j)-\Capa_p(B_1)}{\mathcal{A}(\Omega_j)^2}$ converges to 0.
	We then perturb the sequence $\{\Omega_j\}$ in such a way
	that the new sequence $\{\tilde{\Omega}_j\}$ still contradicts
	the theorem but also consists of minimizers of certain
	free boundary problems. By employing regularity results
	of \cite{DP}, we infer that the sets $\tilde{\Omega}_j$ need to be
	smooth. That leads us to contradiction with the first step of 
	the proof. This is the content of Sections \ref{penpb}-\ref{reg}.
	
	Finally, in Sections \ref{redtobdd} and \ref{s:proof} we reduce
	the general case to the one of bounded sets of small asymmetry.
	To reduce to the bounded sets, we provide bounds for the asymmetry
	and deficit of the set $\Omega\cap B_R$, where $\Omega$ is an
	arbitrary open set and $R$ is big enough. To reduce to the sets
	of small asymmetry we use the non-sharp quantitative inequality
	of \cite{nonsharp}.
	
	Note that as in \cite{fkstab} and \cite{DPMM}, we replace the Fraenkel asymmetry (which roughly resembles a \(L^1\) type norm) with a smoother (and stronger) version inspired by the distance among sets first used by Almgren, Taylor, and Wang in~\cite{AlmgrenTaylorWang93} which resembles an \(L^2\) type norm, see Section \ref{sec:lin} for the exact definition.	

 \subsection*{Acknowledgements}
The work of the author is supported by the INDAM-grant ``Geometric Variational Problems" and by the H2020-MSCA-RISE-2017 PROJECT No. 778010 IPADEGAN. The author wishes to thank Guido De Philippis for many fruitful discussions.

	\section{Fuglede's computation}	\label{fuglede}

		In this section we are going to prove the validity of the  quantitative isocapacitary inequality for sets close to the unit ball.
		More precisely, we are going to prove Theorem \ref{thm:main p-cap} for nearly spherical sets which are defined below. The proof is based on second variation argument as in \cite{Fuglede89} (compare with \cite[Section 2]{DPMM}).		

		\begin{defin}\label{d:ns}
		  An open bounded set $\Omega\subset\mathds{R}^N$ is called nearly spherical
			of class $C^{2,\gamma}$ parametrized by $\varphi$, if there exists
			$\varphi\in C^{2,\gamma}$ with $\Vert\varphi\Vert_{L^\infty}<\frac{1}{2}$ such that
			\begin{equation*}
				\partial\Omega=\{(1+\varphi(x))x: x\in\partial B_1\}.
			\end{equation*}	
		\end{defin}

	As mentioned in the introduction, we are going to work with another notion of asymmetry. Let us define it and state basic properties.	
	
	\begin{defin}
		Let $\Omega$ be an open set in $\mathds{R}^N$.
		Then we define the asymmetry $\alpha$ in the following way:
				\begin{equation*}
					\alpha(\Omega)=\int_{\Omega\Delta B_1(x_\Omega)}\big\vert 1-\vert x-x_\Omega\vert\big\vert dx.
				\end{equation*}
		Here $x_\Omega$ denotes the barycenter of $\Omega$, namely $x_\Omega=\fint_{\Omega}{x}dx$.
	\end{defin}	
	
	\begin{lemma}[{\cite[Lemma 4.2]{fkstab}}] \label{propasym}
	Let \(\Omega \subset \mathds R^n\), then 
		\begin{enumerate}[label=\textup{(\roman*)}]
			\item \label{compasym} There exists a constant $c=c(N)$ such that
							$$\alpha(\Omega)\geq c\vert\Omega\Delta B_1(x_\Omega)\vert^2$$
						for any open set $\Omega$.
			\item \label{asymlip} There exists a constant $C=C(R)$ such that	
				\begin{equation*}
					\vert\alpha(\Omega_1)-\alpha(\Omega_2)\vert\leq C\vert\Omega_1\Delta \Omega_2\vert
				\end{equation*}
				for any $\Omega_1,\Omega_2\subset B_R$. In particular, if $1_{\Omega_k}\rightarrow 1_\Omega$ in $L^1(B_R)$
				then  $\alpha(\Omega_k)\rightarrow\alpha(\Omega)$.
			\item \label{asymnrlysphrsets}
			There exist constants $C=C(N)$, $\delta=\delta(N)$ such that
			for every nearly spherical set $\Omega$ parametrized by $\varphi$ with $\Vert\varphi\Vert_{\infty}\leq\delta$ 
			and $x_\Omega=0$
			\begin{equation*}
				\alpha(\Omega)\leq C\Vert\varphi\Vert_{L^2(\partial B_1)}^2.
			\end{equation*}	
		\end{enumerate}
	\end{lemma}
	We are also going to use the following norm:
	\[
	\Vert\varphi\Vert^2_{H^\frac{1}{2}(\partial B_1)}:=\int_{\partial B_1}\varphi^2 d\mathcal{H}^{N-1}+\int_{B^c_1}\vert\nabla H(\varphi)\vert^2 dx,
	\]
	where $H(\varphi)$ is the harmonic extension of $\varphi$ into $B^c$, i.e. it is the solution of
			\begin{equation*}
				\begin{cases}
					\Delta H(\varphi)=0\text{ in }B_1^c, \\
					H(\varphi)=\varphi\text{ on }\partial B_1, \\
					H(\varphi)(x)\rightarrow 0\text{ as }\vert x\vert\rightarrow\infty
				\end{cases}
			\end{equation*}	
	in $D^{1,2}(\mathds{R}^N)$. Note that this norm is equivalent to the standard one, where the second integral is replaced by Gagliardo seminorm (see for example \cite[(1,3,3,3)]{Gr}).
		
	We now want to write Taylor expansion for $p$-capacity around the ball. Let us first introduce a technical lemma that is going to give us a sequence of sets along which we will be taking derivatives.
One can find its proof in \cite[Lemma A.1]{fkstab}.
		
		\begin{lemma}\label{lm:vectorfield}
			Given $\gamma\in(0,1]$ there exists $\delta=\delta(N,\gamma)>0$
			and a modulus of continuity $\omega$ such that for every nearly spherical set $\Omega$
			parametrized by $\varphi$ with $\Vert\varphi\Vert_{C^{2,\gamma}(\partial B_1)}<\delta$
			and $\vert\Omega\vert=\vert B_1\vert$, we can find an autonomous vector field $X_\varphi$
			for which the following holds true:
			\begin{enumerate}[label=\textup{(\roman*)}]
				\item $\div{X_\varphi}=0$ in a $\delta$-neighborhood of $\partial B_1$;
				\item if $\Phi_t:=\Phi(t,x)$ is the flow of $X_\varphi$, i.e.
					\begin{equation*}
					  \partial_t\Phi_t=X_\varphi(\Phi_t),\qquad \Phi_0(x)=x,
					\end{equation*}
					then $\Phi_1(\partial B_1)=\partial\Omega$ and $\vert\Phi_t(B_1)\vert=\vert B_1\vert$ for all $t\in[0,1]$;
				\item 
					\begin{itemize}
						\item $\Vert\Phi_t-Id\Vert_{C^{2,\gamma}}\leq\omega(\Vert\varphi\Vert_{C^{2,\gamma}(\partial B_1)})$ for every $t\in[0,1]$,
						\medskip
						
						\item $\Vert\varphi-(X_\varphi\cdot\nu_{B_1})\Vert_{H^{\frac{1}{2}}(\partial B_1)}\leq\omega(\Vert\varphi\Vert_{L^\infty(\partial B_1)})\Vert\varphi\Vert_{H^{\frac{1}{2}}(\partial B_1)}$,
						
						\medskip
						
						\item $(X\cdot x)\circ\Phi_t-X\cdot\nu_{B_1}=(X\cdot\nu_{B_1})\psi_t$, $x\in\partial B_1$, 
						
						where $\Vert\psi_t\Vert_{C^{2,\gamma}(\partial B_1)}\leq\omega(\Vert\varphi\Vert_{C^{2,\gamma}(\partial B_1)})$.
					\end{itemize}
			\end{enumerate}
			
		\end{lemma}
		
		\begin{rem}
		Here and in the sequel we will denote by $\nu_\Omega$ the \emph{inner} unit normal.
		\end{rem}

		Following \cite{FZ} (note
that we will be citing the preprint rather than the published version \cite{FZpublished} as it has
more details), we consider perturbed functionals
		to make the equation non-degenerate.
		
		\begin{defin} We define perturbed $p$-capacity as follows:
			\begin{equation*}
				\mathrm{cap}_{p,\kappa}(\Omega)=\inf_{u\in W^{1,p}(\mathbb{R}^N)}\left\{\int_{\Omega^c}\left((\kappa^2+\vert\nabla u\vert^2)^{\frac{p}{2}}-\kappa^p\right)\,dx: u=1 \text{ on }\Omega\right\}.
			\end{equation*}
		\end{defin}
		
		\begin{rem}
			Note that the infimum is achieved by the unique solution of the following equation
			\begin{equation*}
				\begin{cases}
					\div((\kappa^2+\vert\nabla u\vert^2)^{\frac{p-2}{2}}\nabla u)=0\text{ in }\Omega^c,\\
					u=1\text{ on }\partial\Omega,\\
					u(x)\rightarrow 0\text{ as }\vert x\vert\rightarrow\infty.
				\end{cases}
			\end{equation*}
			We will denote the minimizer by $u_{\kappa,\Omega}$.
		\end{rem}

		Let $\Phi_t$ be the flow from the lemma and define $\Omega_t=\Phi_t(B)$. For brevity we denote $u_{\kappa,t}:=u_{\kappa,\Omega_t}$.
		\begin{rem}\label{rem:pcap eq for pert energy and direction of the norm}
			The function $u_{\kappa,t}$ satisfies the following equation
			\begin{equation}\label{e:pcap eq kappa pert energy}
				\begin{cases}
					\div((\kappa^2+\vert\nabla u_{\kappa,t}\vert^2)^{\frac{p-2}{2}}\nabla u_{\kappa,t})=0\text{ in }\Omega_t^c,\\
					u_{\kappa,t}=1\text{ on }\partial\Omega_t,\\
					u_{\kappa,t}(x)\rightarrow 0\text{ as }\vert x\vert\rightarrow\infty.
				\end{cases}
			\end{equation}
		We also note that $\nabla u_{\kappa,t}=\vert\nabla u_{\kappa,t}\vert\nu_{\partial\Omega_t}$ 
		on $\partial\Omega_t$ since it is constant on the boundary and
		less than $1$ outside of the set by maximum principle.
		\end{rem}

		We want to differentiate the perturbed $p$-capacity of $\Omega_t$ in $t$. We introduce the following notation
		\begin{equation*}
			c_k(t):=\int_{\Omega_t^c}\left((\kappa^2+\vert\nabla u_{\kappa,t}\vert^2)^{\frac{p}{2}}-\kappa^p\right)\,dx.
		\end{equation*}
		Since for any $\kappa>0$ equation \eqref{e:pcap eq kappa pert energy} is elliptic, the following standard (see e.g. \cite[Proposition 3.1]{SZ}) differentiability result holds.
		\begin{lemma}[Shape derivative of $u_{\kappa,t}$]
		For any $\kappa>0$ the derivative of $u_{\kappa,t}$ in $t$ exists 
		and it solves the following equation
		\begin{equation}\label{e:eq for u dot pcap}
			\begin{cases}
				\div\big((\kappa^2+\vert\nabla u_{\kappa,t}\vert^2)^\frac{p-2}{2}\nabla \dot{u}_{\kappa,t}\\
				\quad\quad+(p-2)(\kappa^2+\vert\nabla u_{\kappa,t}\vert^2)^\frac{p-4}{2}(\nabla u_{\kappa,t}\cdot\nabla\dot{u}_{\kappa,t})\nabla u_{\kappa,t}\big)=0\text{ in }\Omega_t^c,\\
				\dot{u}_{\kappa,t}=-\nabla u_{\kappa,t}\cdot X\text{ on }\partial\Omega_t.
			\end{cases}
		\end{equation}
		\end{lemma}

		We want to see what happens near the ball for the initial functional. 
		To that end, we compute $u_0:=u_{0,0}$ and its gradient:
		\begin{equation*}
			u_0=\vert x\vert^\frac{p-n}{p-1}, \nabla u_0=\frac{p-n}{p-1}\vert x\vert^{\frac{p-n}{p-1}-1}\theta,
		\end{equation*}				
		where $\theta:=\frac{x}{\vert x\vert}$.

		\begin{theorem}[convergence of $u_{\kappa,t}$]\label{thm:convergence of u_kappa to u_0}
		Let $\kappa\in [0,1]$, $p>1$, $\alpha\in (0,1)$, $R>1$. 
		There exist $\beta\in(0,\alpha)$ and a modulus of continuity
		$\omega=\omega(p, \alpha, n)$
		such that if $\Omega$ is a $C^{2,\alpha}$ 
		nearly spherical set parametrized by $\varphi$ and 
		$\Vert\varphi\Vert_{C^{2,\alpha}(\partial B)}<\delta$, 
		then for all $t\in [0,1]$ and $\kappa\in [0,1]$
		we have
		\begin{equation*}
			\Vert u_0-u_{\kappa,t}\circ\Phi_t\Vert_{C^{1,\beta}(B_R\backslash B_1)}\leq\omega(\Vert\varphi\Vert_{C^{2,\alpha}(\partial B)}+\kappa).
		\end{equation*}				
		
		Moreover, there exist $\delta'>0$, 
		$0<\gamma<\alpha$ and a modulus of continuity 
		$\omega'=\omega'(p,\alpha,n,\varepsilon)$, such that if 
		$\Vert\varphi\Vert_{C^{2,\alpha}(\partial B)}+\kappa<\delta'$, 
		then for all $t\in [0,1]$
		\begin{equation*}
			\Vert u_0-u_{\kappa,t}\circ\Phi_t\Vert_{C^{2,\gamma}(B_R\backslash B_1)}\leq\omega'(\Vert\varphi\Vert_{C^{2,\alpha}(\partial B)}+\kappa).
		\end{equation*}				
		\end{theorem}
		\begin{proof}
			The proof goes in the same way as the one of \cite[Theorem 2.2]{FZ}.
			We reproduce it here for the reader's convenience.
			
			First, we notice that regularity for degenerate elliptic equations (see \cite[Theorem 1]{L}) gives us
			\begin{equation}\label{e:C^1,beta bound for solutions pcap}
			\Vert u_{\kappa,t}\Vert_{C^{1,\beta'}(B_R\setminus\Omega_t)}\leq C=C(p,\alpha,n,\delta)
			\end{equation}
			for some $\beta'\in(0,\alpha)$, and every $\kappa\in[0,1]$,
			$t\in[0,1]$. Fix $\beta\in(0,\beta')$.
			To prove the first inequality we argue by contradiction. Suppose there exist sequences $\{\varphi_j\}$,
			$\{\kappa_j\}$, $\{t_j\}$ such that $\Vert\varphi_j\Vert_{C^{2,\alpha}(\partial B)}+\kappa_j\rightarrow 0$,
			\begin{equation}\label{e:contr assumption convergence of u_kappa}
			\limsup_{j\rightarrow\infty} \Vert u_0-u_{\kappa_j,t_j}\circ\Phi^j_{t_j}\Vert_{C^{1,\beta}(B_R\backslash B_1)}>0,
			\end{equation}
			where $\Phi^j$ is the flow from Lemma \ref{lm:vectorfield} associated with $\varphi_j$. 
			Using \eqref{e:C^1,beta bound for solutions pcap}, we
			extract a (non-relabelled) subsequence such that
			$\tilde{u}_j:=u_{\kappa_j,t_j}\circ\Phi^j_{t_j}$ converges
			to a function $u$ in $C^{1,\beta}$.	Each function 
			$\tilde{u}_j$ satisfies
			\begin{equation*}
			\begin{cases}
				\div\left(\left(\kappa_j^2+\left\vert\left(\left(\nabla\Phi^j_{t_j}\right)^{-1}\right)^t\nabla \tilde{u}_j\right\vert^2\right)^\frac{p-2}{2}M_j\nabla \tilde{u}_j\right)=0\text{ in } B^c,\\
				\tilde{u}_j=1\text{ on }\partial B,\\
				\tilde{u}_j(x)\rightarrow 0\text{ as }\vert x\vert\rightarrow \infty,
			\end{cases}
			\end{equation*}					
			where $M_j=\mathrm{det}\nabla\Phi^j_{t_j}\left(\nabla\Phi^j_{t_j}\right)^{-1}\left(\left(\nabla\Phi^j_{t_j}\right)^{-1}\right)^t$. Thus, $u$, as a limit of $\tilde{u}_j$ in $C^{1,\beta}$,
			satisfies
			\begin{equation*}
			\begin{cases}
				\div\left(\left\vert\nabla u\right\vert^{p-2}\nabla u\right)=0\text{ in } B^c,\\
				u=1\text{ on }\partial B,\\
				u(x)\rightarrow 0\text{ as }\vert x\vert\rightarrow\infty,
			\end{cases}
			\end{equation*}					
			meaning that $u$ coincides with $u_0$, which contradicts
			\eqref{e:contr assumption convergence of u_kappa}.
						
			To get convergence in $C^{2,\gamma}$, we notice
			that $0<c\leq\vert\nabla u_0\vert\leq C$ in $B_R\backslash B_1$. The $C^{1,\beta'}$ converges gives us that the same is true for $\nabla u_{\kappa,t}$ if $\Vert\varphi\Vert_{C^{2,\alpha}(\partial B)}+\kappa$ is small enough. From here equation for $u_{\kappa,t}$
			and Schauder estimates give us
			\begin{equation*}
			\Vert u_{\kappa,t}\Vert_{C^{2,\gamma'}(B_R\setminus\Omega_t)}\leq C=C(p,\alpha,n,\delta).
			\end{equation*}
			We now can argue in the same way as we did for $C^{1,\beta'}$ convergence.
		\end{proof}

		\subsection{First derivative}
		\begin{prop}
			For $\kappa>0$, $t\in[0,1]$ the perturbed p-capacity 
			is differentiable in $t$ and the following formula holds
			\begin{equation}\label{e:first derivative p-capacity}
			\begin{split}
				c'_\kappa(t)&=-p\int_{\partial\Omega_t}{\vert\nabla u_{\kappa,t}\vert^2(\kappa^2+\vert\nabla u_{\kappa,t}\vert^2)^\frac{p-2}{2}(X\cdot\nu)}d\mathcal{H}^{N-1}\\
				&\qquad+\int_{\partial\Omega_t}{\left((\kappa^2+\vert\nabla u_{\kappa,t}\vert^2)^\frac{p}{2}-\kappa^p\right)(X\cdot\nu)}d\mathcal{H}^{N-1}\\
				&=-p\int_{\Omega_t^c}{\div\left(\vert\nabla u_{\kappa,t}\vert^2(\kappa^2+\vert\nabla u_{\kappa,t}\vert^2)^\frac{p-2}{2}X\right)}dx\\
				&\qquad+\int_{\Omega_t^c}{\div\left(\left((\kappa^2+\vert\nabla u_{\kappa,t}\vert^2)^\frac{p}{2}-\kappa^p\right)X\right)}dx.
			\end{split}	
			\end{equation}
			Moreover, for every $t\in[0,1]$ we also have					
			\begin{equation}\label{e:first derivative at zero p-capacity}
				c'_0(t)=-(p-1)\int_{\partial\Omega_t}{\vert\nabla u_{0,t}\vert^p(X\cdot\nu)}d\mathcal{H}^{N-1}.
			\end{equation}
		\end{prop}
		\begin{proof}
			By Hadamard's formula (see \cite[Chapter 5]{HP}),
			\begin{equation*}
			\begin{aligned}
				c'_\kappa(t)&=\int_{\Omega_t^c}{p\nabla u_{\kappa,t}\cdot\nabla \dot{u}_{\kappa,t}(\kappa^2+\vert\nabla u_{\kappa,t}\vert^2)^\frac{p-2}{2}}\,dx+\int_{\partial\Omega_t}{\left((\kappa^2+\vert\nabla u_{\kappa,t}\vert^2)^\frac{p}{2}-\kappa^p\right)(X\cdot\nu)}d\mathcal{H}^{N-1}\\
				&=\int_{\partial\Omega_t}{\left(-\nabla u_{\kappa,t}\cdot X\right)p(\kappa^2+\vert\nabla u_{\kappa,t}\vert^2)^\frac{p-2}{2}}\nabla u_{\kappa,t}\cdot\nu\,d\mathcal{H}^{N-1}\\
				&\qquad+\int_{\partial\Omega_t}{\left((\kappa^2+\vert\nabla u_{\kappa,t}\vert^2)^\frac{p}{2}-\kappa^p\right)(X\cdot\nu)}d\mathcal{H}^{N-1},
			\end{aligned}	
			\end{equation*}						
			where for the second equality we used the equations
			\eqref{e:pcap eq kappa pert energy} and \eqref{e:eq for u dot pcap}.
			It remains to notice that $\nabla u_{\kappa,t}=\vert\nabla u_{\kappa,t}\vert\nu$ on $\partial\Omega_t$ as noted in Remark
			\ref{rem:pcap eq for pert energy and direction of the norm}. This gives us the first equality
			of \eqref{e:first derivative p-capacity}, whereas the second equality of \eqref{e:first derivative p-capacity} follows
			from divergence theorem.			
			
			The convergence established in Theorem \ref{thm:convergence of u_kappa to u_0}  gives us \eqref{e:first derivative at zero p-capacity}.
		\end{proof}
		\subsection{Second derivative}
			To state the results for the second derivative we
			need to introduce the following weighted Sobolev space:
			$$D^{1,2}(B^c,\mu):=\left\{u\in H^1_{loc}(B^c):\int_{B^c}{\vert\nabla u\vert^2\,d\mu}<\infty\right\},$$
			where $d\mu=\vert x\vert^{\left(\frac{p-n}{p-1}-1\right)(p-2)}dx$.
			We denote by $D^{1,2}_0(B^c,\mu)$ the corresponding space of functions with zero trace on $\partial B$. 		
			
			\noindent
			We also denote the mean curvature of a smooth set $\Omega$ by $\mathscr{H}_{\Omega}$ ($\mathscr{H}_{\Omega}=\div\,\nu_{\Omega}${, note that $\mathscr{H}_{\Omega_t^c}=-\mathscr{H}_{\Omega_t}$}).
						
			\begin{prop}\label{prop:second derivative p-capacity}
			We define $X_\tau:=X-\left(X\cdot\nu\right)\nu$. Then
			for $\kappa>0, t\in[0,1]$ the perturbed $p$-capacity
			is twice differentiable and the following formula holds:
			\begin{equation}\label{e:second derivative p-capacity}
			\begin{aligned}
				\frac{1}{p}c''_\kappa(t)&=\int_{\partial\Omega_t}{(\kappa^2+\vert\nabla u_{\kappa,t}\vert^2)^\frac{p-2}{2}(\nabla\dot{u}_{\kappa,t}\cdot\nu)\dot{u}_{\kappa,t}}d\mathcal{H}^{N-1}\\
				&+(p-2)\int_{\partial\Omega_t}{(\kappa^2+\vert\nabla u_{\kappa,t}\vert^2)^\frac{p-4}{2}
				(\nabla\dot{u}_{\kappa,t}\cdot\nabla u_{\kappa,t})(\nabla\dot{u}_{\kappa,t}\cdot\nu)}\dot{u}_{\kappa,t}d\mathcal{H}^{N-1}\\
				&-\int_{\partial\Omega_t}{(\kappa^2+\vert\nabla u_{\kappa,t}\vert^2)^\frac{p-2}{2}(\nabla^2 u_{\kappa,t}[\nabla u_{\kappa,t}]\cdot X_\tau)(X\cdot\nu)}d\mathcal{H}^{N-1}\\
				&-(p-2)\int_{\partial\Omega_t}{(\kappa^2+\vert\nabla u_{\kappa,t}\vert^2)^\frac{p-4}{2}\vert\nabla u_{\kappa,t}\vert^2(\nabla^2 u_{\kappa,t}[\nabla u_{\kappa,t}]\cdot X_\tau)(X\cdot\nu)}d\mathcal{H}^{N-1}\\
				&+\int_{\partial\Omega_t}{\vert\nabla u_{\kappa,t}\vert^2(X\cdot\nu)^2(\kappa^2+\vert\nabla u_{\kappa,t}\vert^2)^\frac{p-2}{2}\mathscr{H}_{\Omega_t^c}}d\mathcal{H}^{N-1}.
			\end{aligned}	
			\end{equation}
			Moreover, 
			\begin{equation*}
			\begin{split}
				\left(\frac{p-1}{N-p}\right)^{p-2}\frac{1}{p}c''_0(0&)=-(N-1)\int_{\partial B}{\dot{u}_0^2}d\mathcal{H}^{N-1}\\
				&+\int_{B^c}{\vert x\vert^{(p-2)\left(\frac{p-N}{p-1}-1\right)}\left(\vert\nabla\dot{u}_0\vert^2+(p-2)(\theta\cdot\nabla\dot{u}_0)^2\right)}dx,
			\end{split}	
			\end{equation*}
			\end{prop}
			where $\dot{u}_0$ solves
			\begin{equation}\label{e:el for u dot}
			\begin{cases}
				\div\left(\vert x\vert^{(p-2)(\frac{p-N}{p-1}-1)}\nabla \dot{u}_0+(p-2)\vert x\vert^{(p-2)(\frac{p-N}{p-1}-1)}(\theta\cdot\nabla\dot{u}_0)\theta\right)=0\text{ in }B^c,\\
				\dot{u}_0=\frac{N-p}{p-1}\theta\cdot X\text{ on }\partial B
			\end{cases}
			\end{equation}
			in $W^{1,2}(B^c,d\mu)$.
			\begin{proof}

			{\bf Computation.}			
			First we use Hadamard's formula to differentiate
			the second equality of \eqref{e:first derivative p-capacity}.
			Using then the divergence theorem, we get
			\begin{equation*}
			\begin{split}
				&c_\kappa''(t)=\int_{\partial\Omega_t}{p\left(\kappa^2+\vert\nabla u_{\kappa,t}\vert^2\right)^{(p-2)/2}\left(\nabla u_{\kappa,t}\cdot\nabla\dot{u}_{\kappa,t}\right)\left(X\cdot\nu\right)}\,d\mathcal{H}^{N-1}\\
				&+\int_{\partial\Omega_t}{\div\left(\left((\kappa^2+\vert\nabla u_{\kappa,t}\vert^2)^\frac{p}{2}-\kappa^p\right)X\right)\left(X\cdot\nu\right)}\,d\mathcal{H}^{N-1}\\
				&-p\int_{\partial\Omega_t}{\left((p-2)\kappa^2+\vert\nabla u_{\kappa,t}\vert^2\right)^{(p-4)/2}\vert\nabla u_{\kappa,t}\vert^2\left(\nabla u_{\kappa,t}\cdot\nabla\dot{u}_{\kappa,t}\right)\left(X\cdot\nu\right)}\,d\mathcal{H}^{N-1}\\
				&-2p\int_{\partial\Omega_t}{2\left(\kappa^2+\vert\nabla u_{\kappa,t}\vert^2\right)^{(p-2)/2}\left(\nabla u_{\kappa,t}\cdot\nabla\dot{u}_{\kappa,t}\right)\left(X\cdot\nu\right)}\,d\mathcal{H}^{N-1}\\
				&-p\int_{\partial\Omega_t}{\div\left(\left(\kappa^2+\vert\nabla u_{\kappa,t}\vert^2\right)^{(p-2)/2}\vert\nabla u_{\kappa,t}\vert^2 X\right)\left(X\cdot\nu\right)}\,d\mathcal{H}^{N-1}.\\				
			\end{split}	
			\end{equation*}
			Using \eqref{e:eq for u dot pcap}, and the fact that $X$ is divergence-free, we obtain
			\begin{equation}\label{e:second derivative - first equation}
			\begin{split}
				&c_\kappa''(t)=p\int_{\partial\Omega_t}{\left(\kappa^2+\vert\nabla u_{\kappa,t}\vert^2\right)^{(p-2)/2}\dot{u}_{\kappa,t}\left(\nabla \dot{u}_{\kappa,t}\cdot\nu\right)}\,d\mathcal{H}^{n-1}\\
				&-\int_{\partial\Omega_t}{p\left(\kappa^2+\vert\nabla u_{\kappa,t}\vert^2\right)^{(p-2)/2}\left(\nabla^2 u_{\kappa,t}[\nabla u_{\kappa,t}]\cdot X\right)\left(X\cdot\nu\right)}\,d\mathcal{H}^{n-1}\\
				&+p\int_{\partial\Omega_t}{\left((p-2)\left(\kappa^2+\vert\nabla u_{\kappa,t}\vert^2\right)^{(p-4)/2}\vert\nabla u_{\kappa,t}\vert^2\right)\dot{u}_{\kappa,t}\left(\nabla \dot{u}_{\kappa,t}\cdot\nu\right)}\\
				&-p(p-2)\int_{\partial\Omega_t}{\left(\kappa^2+\vert\nabla u_{\kappa,t}\vert^2\right)^{(p-4)/2}\vert\nabla u_{\kappa,t}\vert^2\left(\nabla^2 u_{\kappa,t}[\nabla u_{\kappa,t}]\cdot X\right)\left(X\cdot\nu\right)}\,d\mathcal{H}^{n-1}.				
			\end{split}	
			\end{equation}												
			We now use Remark \ref{rem:pcap eq for pert energy and direction of the norm} to get
			the following equality on the boundary
			\begin{equation}\label{e:divergence on the boundary}
			\begin{aligned}
				0&=\div\left((\kappa^2+\vert\nabla u_{\kappa,t}\vert^2)^\frac{p-2}{2}\nabla u_{\kappa,t}\right)\\
				&=(\kappa^2+\vert\nabla u_{\kappa,t}\vert^2)^\frac{p-2}{2}\vert\nabla u_{\kappa,t}\vert\mathscr{H}_{\Omega_t^c}+(\kappa^2+\vert\nabla u_{\kappa,t}\vert^2)^\frac{p-2}{2}\nabla^2 u_{\kappa,t}[\nu]\cdot\nu\\
				&+(p-2)(\kappa^2+\vert\nabla u_{\kappa,t}\vert^2)^\frac{p-4}{2}\vert\nabla u_{\kappa,t}\vert^2\nabla^2 u_{\kappa,t}[\nu]\cdot\nu.
			\end{aligned}	
			\end{equation}
			Now we plug \eqref{e:divergence on the boundary} into
			\eqref{e:second derivative - first equation} and get
			\eqref{e:second derivative p-capacity}.
			
			\noindent						
			{\bf Convergence.}
			Fix $R>1$. By Schauder estimates functions $u_{\kappa,t}\circ\Phi_t$
			are equibounded in $C^{2,\gamma}(B_R\backslash B)$
			and $\vert\nabla u_{\kappa,t}\vert\in (c(R), C(R))$
			for $\kappa$, $t$ small.
			Thus, from \eqref{e:eq for u dot pcap}, using classical elliptic
			estimates we get that $\dot{u}_{\kappa,t}\circ\Phi_t$ are equibounded in $C^{1,\gamma}(B_R\backslash B)$ and up to a subsequence converge to a function 
			$\hat{w}\in C^1(B^c)$ uniformly on compacts.
			
			Using $\dot{u}_{\kappa,t}$ as a test function in \eqref{e:eq for u dot pcap} and applying
			divergence theorem, we get
			\begin{equation*}
			\begin{split}
			&\int_{\Omega_t^c}{(\kappa^2+\vert\nabla u_{\kappa,t}\vert^2)^\frac{p-2}{2}\left\vert\nabla \dot{u}_{\kappa,t}\right\vert^2}dx
			+(p-2)\int_{\Omega_t^c}{(\kappa^2+\vert\nabla u_{\kappa,t}\vert^2)^\frac{p-4}{2}(\nabla u_{\kappa,t}\cdot\nabla\dot{u}_{\kappa,t})^2}dx\\
			&=\int_{\partial\Omega_t^c}{(\kappa^2+\vert\nabla u_{\kappa,t}\vert^2)^\frac{p-2}{2}\dot{u}_{\kappa,t}\nabla \dot{u}_{\kappa,t}\cdot\nu}\\
			&\qquad+\int_{\partial\Omega_t^c}{(p-2)(\kappa^2+\vert\nabla u_{\kappa,t}\vert^2)^\frac{p-4}{2}(\nabla u_{\kappa,t}\cdot\nabla\dot{u}_{\kappa,t})\dot{u}_{\kappa,t}\nabla u_{\kappa,t}\cdot\nu}\leq C.
			\end{split}
			\end{equation*}
			That means that $\hat{w} \in D^{1,2}(B^c,\mu)$.		
			Passing to the limit in \eqref{e:eq for u dot pcap} as $(\kappa,t)\rightarrow (0,0)$, we get
			\begin{equation}\label{e:pcap eq for w hat}
				\int_{B^c}{\nabla\hat{w}\cdot\nabla v+(p-2)(\theta\cdot\nabla\hat{w})(\theta\cdot\nabla v)}d\mu=0
			\end{equation}
			for any $v\in W^{1,2}(B_R\backslash\overline{B})$
			with compact support in $B_R\backslash\overline{B}$.
			
			It remains to show that the same identity holds for every 
			$v$ in $D^{1,2}_0(B^c;\mu)$ and that $\hat{w}$ is the unique
			solution of the equation \eqref{e:pcap eq for w hat} in $D^{1,2}(B^c;\mu)$ with a
			proper boundary condition. To that end, we fix $R>1$
			and a cut-off function $\eta_R$ such that $\eta_R\equiv 1$
			in $B_R\backslash B$, $\eta_R\equiv 0$ in $B_{2R}^c$ and 
			$\vert\nabla\eta_R\vert\leq C/R$. 
			We also fix a constant $c$ that we specify later, it will depend on $R$.
			For a function $v\in D^{1,2}_0(B^c;\mu)$ we plug $(v-c)\eta_R$ into \eqref{e:pcap eq for w hat} to get
			\begin{equation*}
			\begin{split}
				&\int_{B^c}{\eta_R\nabla\hat{w}\cdot\nabla v+(p-2)\eta_R(\theta\cdot\nabla\hat{w})(\theta\cdot\nabla v)}d\mu\\
				&=\int_{B^c}{\eta_R\nabla\hat{w}\cdot\nabla (v-c)+(p-2)\eta_R(\theta\cdot\nabla\hat{w})(\theta\cdot\nabla (v-c))}d\mu\\
				&=-\left(\int_{B^c}{(v-c)\nabla\hat{w}\cdot\nabla \eta_R+(p-2)(v-c)(\theta\cdot\nabla\hat{w})(\theta\cdot\nabla\eta_R)}d\mu\right).
			\end{split}	
			\end{equation*}
			
			\noindent
			{\bf Claim:}
$				\int_{B^c}{(v-c)\nabla\hat{w}\cdot\nabla \eta_R+(p-2)(v-c)(\theta\cdot\nabla\hat{w})(\theta\cdot\nabla\eta_R)}d\mu\rightarrow 0\text{ as }R\rightarrow \infty.
$

			\noindent
			Indeed, we have
			\begin{equation*}
			\begin{split}
				&\left\vert\int_{B^c}{(v-c)\nabla\hat{w}\cdot\nabla \eta_R+(p-2)(v-c)(\theta\cdot\nabla\hat{w})(\theta\cdot\nabla\eta_R)}d\mu\right\vert\\
				&=\left\vert\int_{B_{2R}\setminus B_R}{\left((v-c)\nabla\hat{w}\cdot\nabla \eta_R+(p-2)(v-c)(\theta\cdot\nabla\hat{w})(\theta\cdot\nabla\eta_R)\right)\vert x\vert^{\left(\frac{p-N}{p-1}-1\right)(p-2)}}\,dx\right\vert\\
				&\leq (p-1)\frac{C}{R}\int_{B_{2R}\setminus B_R}{\vert v-c\vert\vert\nabla\hat{w}\vert\vert x\vert^{\left(\frac{p-N}{p-1}-1\right)(p-2)}}\,dx\\
				&\leq\frac{C}{R}\left(\int_{B_{2R}\setminus B_R}{\vert\nabla\hat{w}\vert^2\vert x\vert^{\left(\frac{p-N}{p-1}-1\right)(p-2)}}\,dx\right)^{1/2}
				\left(\int_{B_{2R}\setminus B_R}{\vert v-c\vert^2\vert x\vert^{\left(\frac{p-N}{p-1}-1\right)(p-2)}}\,dx\right)^{1/2}. 	
			\end{split}
			\end{equation*}
			Since we know that $w\in D^{1,2}(B^c;\mu)$, if we manage to 
			show that 
			\begin{equation}\label{e:ineq for the test function}
			\frac{1}{R}\left(\int_{B_{2R}\setminus B_R}{\vert v-c\vert^2\vert x\vert^{\left(\frac{p-N}{p-1}-1\right)(p-2)}}\,dx\right)^{1/2}\leq C
			\end{equation} 
			for some $C=C(N,p)$, the claim will be proven. For convenience we denote $\gamma:=\left(\frac{p-N}{p-1}-1\right)(p-2)$. 
			We now choose $c=\fint_{B_{2R}\setminus B_R}{v}$ and use Poincar\'e inequality to get
			\begin{equation}\label{e:poincare for test function}
			\begin{split}
			\int_{B_{2R}\setminus B_R}{\vert v-c\vert^2}\,dx=
			\int_{B_{2R}\setminus B_R}{\left\vert v-\fint_{B_{2R}\setminus B_R}{v}\right\vert^2}\,dx
			\leq CR^{2}\int_{B_{2R}\setminus B_R}{\left\vert \nabla v\right\vert^2}\,dx\\
			=CR^{2-\gamma}\int_{B_{2R}\setminus B_R}{\left\vert \nabla v\right\vert^2 R^\gamma}\,dx
			\leq CR^{2-\gamma},
			\end{split}
			\end{equation}
			where in the last inequality we used that $v\in D^{1,2}(B^c;\mu)$.
			This yields \eqref{e:ineq for the test function}, since
			\begin{equation*}
			\begin{split}
			\frac{1}{R}\left(\int_{B_{2R}\setminus B_R}{\vert v-c\vert^2\vert x\vert^\gamma}\,dx\right)^{1/2}
			\leq C R^{-1+\gamma/2}\left(\int_{B_{2R}\setminus B_R}{\vert v-c\vert^2}\,dx\right)^{1/2}
			\leq C,
			\end{split}
			\end{equation*}
			where in the last inequality we applied \eqref{e:poincare for test function}. Thus we proved the claim
			and so \eqref{e:pcap eq for w hat} holds for every $v\in D_0^{1,2}(B^c;\mu)$. The boundary condition is 
			the same as in \eqref{e:el for u dot} since $\dot{u}_{\kappa,t}\circ \Phi_t$ are converging to $\hat{w}$ regularly
			on $\partial B$. So, $\hat{w}$ solves the Dirichlet problem \eqref{e:el for u dot} and by uniqueness 
			the whole sequence $\dot{u}_{\kappa,t}\circ \Phi_t$ converges to $\hat{w}$.

		\end{proof}									

		\begin{lemma}\label{l:continuity of the second derivative p cap}
			There exists a modulus of continuity $\omega$ such that 
			\begin{equation*}
				\vert c''_\kappa(t)-c''_\kappa(0)\vert\leq\omega(\Vert\varphi\Vert_{C^{2,\alpha}}+\kappa)\Vert X\cdot\nu\Vert^2_{H^{1/2}(\partial B)}.
			\end{equation*}
		\end{lemma}
		\begin{proof}
			By divergence theorem and \eqref{e:eq for u dot pcap},  using change of variables
			we can rewrite the second derivative of the energy in the following way:
			\begin{equation*}
				\frac{1}{p}c''_\kappa(t)=I_1(t)+I_2(t)+I_3(t),
			\end{equation*}
			where
			\begin{equation*}
			\begin{split}
				&I_1(t):=\int_{B}{\left((\kappa^2+\vert\nabla u_{\kappa,t}\vert^2)^\frac{p-2}{2}\vert\nabla\dot{u}_{\kappa,t}\vert^2\right)\circ \Phi_t\,\mathrm{det}\nabla\Phi_t}dx\\
				&+(p-2)\int_{B}{\left((\kappa^2+\vert\nabla u_{\kappa,t}\vert^2)^\frac{p-4}{2}
				(\nabla\dot{u}_{\kappa,t}\cdot\nabla u_{\kappa,t})^2\right)\circ \Phi_t\,\mathrm{det}\nabla\Phi_t}dx,\\
				&I_2(t):=-\int_{\partial B}{\left((\kappa^2+\vert\nabla u_{\kappa,t}\vert^2)^\frac{p-2}{2}(\nabla^2 u_{\kappa,t}[\nabla u_{\kappa,t}]\cdot X_\tau)(X\cdot\nu)\right)\circ\Phi_t\,J^{\partial B}\Phi_t}d\mathcal{H}^{N-1}\\
				&-(p-2)\int_{\partial B}{\left((\kappa^2+\vert\nabla u_{\kappa,t}\vert^2)^\frac{p-4}{2}\vert\nabla u_{\kappa,t}\vert^2(\nabla^2 u_{\kappa,t}[\nabla u_{\kappa,t}]\cdot X_\tau)(X\cdot\nu)\right)\circ\Phi_t\,J^{\partial B}\Phi_t}d\mathcal{H}^{N-1},\\
				&I_3(t):=\int_{\partial B}{\left(\vert\nabla u_{\kappa,t}\vert^2(X\cdot\nu)^2(\kappa^2+\vert\nabla u_{\kappa,t}\vert^2)^\frac{p-2}{2}\mathscr{H}_{\Omega_t^c}\right)\circ\Phi_t\,J^{\partial B}\Phi_t}d\mathcal{H}^{N-1}.	
			\end{split}
			\end{equation*}
		By Lemma \ref{lm:vectorfield}, we have
		\begin{equation*}
			\Vert\mathscr{H}_{\partial\Omega_t}\circ\Phi_t-\mathscr{H}_{\partial B}\Vert_{L^\infty(\partial B)}+\Vert J_{n-1}\Phi_t-1\Vert_{L^\infty(\partial B)}+\Vert\mathrm{det}\nabla\Phi_t-1\Vert_{L^\infty(\partial B)}\leq\omega\left(\Vert\varphi\Vert_{C^{2,\alpha}}\right).
		\end{equation*}		
		In addition, by Lemma \ref{lm:vectorfield}, $X$ is parallel to
		$\theta$ in a neighborhood of $\partial B$, so we have	
		\begin{equation*}
			\vert\left(X\cdot\nu_{\Omega_t}\right)\circ\Phi_t-X\cdot\nu_B
			\vert
			\leq\omega\left(\Vert\varphi\Vert_{C^{2,\alpha}}\right)\vert X\cdot\nu_B\vert,
		\end{equation*}
		as well as		
		\begin{equation*}
			\vert X_\tau\circ\Phi_t\vert
			\leq\omega\left(\Vert\varphi\Vert_{C^{2,\alpha}}\right)\vert X\cdot\nu_B\vert.
		\end{equation*}
		Thus, using Theorem \ref{thm:convergence of u_kappa to u_0}
		and noticing that $I_2(0)=0$, we get
		\begin{equation*}
			\vert I_2(t)-I_2(0)\vert+\vert I_3(t)-I_3(0)\vert
			\leq\omega\left(\Vert\varphi\Vert_{C^{2,\alpha}}+\kappa\right)\Vert X\cdot\nu_B\Vert^2_{L^2(\partial B)}.
		\end{equation*}
		It remains to show that
		\begin{equation}\label{e:I_1 bound}
			\vert I_1(t)-I_1(0)\vert\leq\omega(\Vert\varphi\Vert_{C^{2,\alpha}}+\kappa)\Vert X\cdot\nu\Vert^2_{H^{1/2}(\partial B)}.
		\end{equation}
		We are going to sketch the proof of \eqref{e:I_1 bound}, for more details see
		the proof of \cite[Lemma 2.7]{FZ}.
		We first move the equation for $\dot{u}_{\kappa,t}$ onto the unit ball $B$. To that end, we denote $w_{\kappa,t}:=\dot{u}_{\kappa,t}\circ\Phi_t$,
		$\tilde{u}_{\kappa,t}:=u_{\kappa,t}\circ\Phi_t$, and
		$N_t=\left(\nabla\Phi_{t}\right)^{-1}\left(\left(\nabla\Phi_{t}\right)^{-1}\right)^t$. Then $w_{\kappa,t}$ satisfies
		\begin{equation*}
			\begin{cases}
				\div\big((\kappa^2+\vert\left(\left(\nabla\Phi_{t}\right)^{-1}\right)^t\nabla \tilde{u}_{\kappa,t}\vert^2)^\frac{p-2}{2}\mathrm{det}\nabla\Phi_t N_t\nabla w_{\kappa,t}\\
				+(p-2)(\kappa^2+\vert\left(\left(\nabla\Phi_{t}\right)^{-1}\right)^t\nabla \tilde{u}_{\kappa,t}\vert^2)^\frac{p-4}{2}\mathrm{det}\nabla\Phi_t(N_t\nabla\tilde{u}_{\kappa,t}\cdot\nabla w_{\kappa,t})N_t\nabla w_{\kappa,t}\big)=0\text{ in }B^c,\\
				w_{\kappa,t}=-\left(\nabla u_{\kappa,t}\cdot X\right)\circ\Phi_t\text{ on }\partial B
			\end{cases}
		\end{equation*}
		and
		\begin{equation*}
		\begin{split}
			I_1(t)&:=\int_{B}{(\kappa^2+\vert\left(\left(\nabla\Phi_{t}\right)^{-1}\right)^t\nabla\tilde{u}_{\kappa,t}\vert^2)^\frac{p-2}{2}N_t\nabla w_{\kappa,t}\cdot\nabla w_{\kappa,t}\,\mathrm{det}\nabla\Phi_t}dx\\
			&+(p-2)\int_{B}{(\kappa^2+\vert\left(\left(\nabla\Phi_{t}\right)^{-1}\right)^t\nabla \tilde{u}_{\kappa,t}\vert^2)^\frac{p-4}{2}
				(N_t\nabla\tilde{u}_{\kappa,t}\cdot\nabla w_{\kappa,t})^2\,\mathrm{det}\nabla\Phi_t}dx.	
		\end{split}
		\end{equation*}
		For convenience we define a bilinear form $L_{\kappa,t,\varphi}$ as
		\begin{equation*}
		\begin{split}
			&L_{\kappa,t,\varphi}(u,v):=\int_{B}{(\kappa^2+\vert\left(\left(\nabla\Phi_{t,\varphi}\right)^{-1}\right)^t\nabla\tilde{u}_{\kappa,t,\varphi}\vert^2)^\frac{p-2}{2}N_{t,\varphi}\nabla u\cdot\nabla v\,\mathrm{det}\nabla\Phi_t}dx\\
			&+(p-2)\int_{B}{(\kappa^2+\vert\left(\left(\nabla\Phi_{t,\varphi}\right)^{-1}\right)^t\nabla \tilde{u}_{\kappa,t,\varphi}\vert^2)^\frac{p-4}{2}
				(N_t\nabla\tilde{u}_{\kappa,t}\cdot\nabla u)(N_t\nabla\tilde{u}_{\kappa,t}\cdot\nabla v)\,\mathrm{det}\nabla\Phi_t}dx,	
		\end{split}
		\end{equation*}
		so that proving \eqref{e:I_1 bound} amounts to showing that
		\begin{equation*}
			\vert L_{\kappa,t,\varphi}(w_{\kappa,t,\varphi},_{\kappa,t,\varphi})-L_{\kappa,0,\varphi}(w_{\kappa,0,\varphi},_{\kappa,t,\varphi})\vert\leq\omega(\Vert\varphi\Vert_{C^{2,\alpha}}+\kappa)\Vert X\cdot\nu\Vert^2_{H^{1/2}(\partial B)}.
		\end{equation*}
		We argue by contradiction. Assume there exist sequences $\kappa_j\rightarrow 0$, $t_j\rightarrow t\in[0,1]$, $\varphi_j\rightarrow 0$ in $C^{2,\alpha}(\partial B)$ such that
		\begin{equation}\label{e:limits of bilinear forms unequal}
			\lim_{j\rightarrow\infty}\frac{L_{\kappa_j,t_j,\varphi_j}(w_{\kappa_j,t_j,\varphi_j},w_{\kappa_j,t_j,\varphi_j})}{\Vert X_j\cdot\nu_B\Vert^2_{H^{1/2}(\partial B)}}
			\neq\lim_{j\rightarrow\infty}\frac{L_{\kappa_j,0,\varphi_j}(w_{\kappa_j,0,\varphi_j},w_{\kappa_j,0,\varphi_j})}{\Vert X_j\cdot\nu_B\Vert^2_{H^{1/2}(\partial B)}}.
		\end{equation}
		Note that we can assume that both limits are finite.
		We define
		\begin{equation*}
		\tilde{w}_j:=\frac{w_{\kappa_j,t_j,\varphi_j}}{\Vert X_j\cdot\nu_B\Vert_{H^{1/2}(\partial B)}},\qquad
				\tilde{w}_{0,j}:=\frac{w_{\kappa_j,0,\varphi_j}}{\Vert X_j\cdot\nu_B\Vert_{H^{1/2}(\partial B)}}.
		\end{equation*}
		One can easily show that $\tilde{w}_j-\tilde{w}_{0,j}\rightarrow 0$
		strongly in $H^{1/2}(\partial B)$.
		A bit more work is required to show that 
		$\tilde{w}_j-\tilde{w}_{0,j}\rightarrow 0$
		strongly in $W^{1,2}(B_R\setminus B)$ for every $r\in(0,1)$.
		To do that, one can prove first that both $\tilde{w}_j$
		and $\tilde{w}_{0,j}$ converge weakly
		to the unique solution in $D^{1,2}(B^c,\mu)$ of
		\begin{equation*}
		\begin{cases}
			\div\left(\vert x\vert^{\frac{p-2}{p-1}}\nabla w+(p-2)\vert x\vert^{\frac{p-2}{p-1}}(\theta\cdot\nabla w)\theta\right)=0\text{ in }B^c,\\
			w=f\text{ on }\partial B,
		\end{cases}	
		\end{equation*}		 
		where $f$ is the weak limit in $H^{1/2}(\partial B)$ of the restriction of $\tilde{w}_j$ on $\partial B$ (remember that the limit of restriction of $\tilde{w}_{0,j}$ is the same). 
		To show the strong convergence consider $z_j$ - the harmonic
		extension of $\tilde{w}_{j}-\tilde{w}_{0,j}$ from $\partial B$
		to $B^c$. Note that $z_j$ converges strongly to zero in $D^{1,2}(B^c)$. Denote by $\zeta\in C^{\infty}_0(B_R)$ a cut-off function such that $\zeta\equiv 1$ on $B_R\setminus B$, $0\leq\zeta\leq 1$. By divergence theorem we get
		\begin{equation*}
		\begin{split}
			&L_{\kappa_j,t_j,\varphi_j}(\tilde{w}_{j}-\tilde{w}_{0,j}, (\tilde{w}_{j}-\tilde{w}_{0,j})\zeta)=L_{\kappa_j,t_j,\varphi_j}(\tilde{w}_j,z_j\zeta)\\
			&-(L_{\kappa_j,t_j,\varphi_j}-L_{\kappa_j,0,\varphi_j})(\tilde{w}_{0,j},(\tilde{w}_j-\tilde{w}_{0,j})\zeta)
			-L_{\kappa_j,0,\varphi_j}(\tilde{w}_{0,j},z_j\zeta)
			\rightarrow 0,
		\end{split}	
		\end{equation*}
		which yields strong convergence of $\tilde{w}_{j}-\tilde{w}_{0,j}$ to zero in $W^{1,2}(B_R\setminus B)$.
		Finally, one can now show that
		\begin{equation*}
			\lim_{j\rightarrow\infty}\left(L_{\kappa_j,t_j,\varphi_j}(\tilde{w}_j,\tilde{w}_j)-L_{\kappa_j,0,\varphi_j}(\tilde{w}_{0,j},
			\tilde{w}_{0,j})\right)=0,			
		\end{equation*}
		contradicting \eqref{e:limits of bilinear forms unequal}.
		\end{proof}

		\begin{lemma}\label{l:taylor for p-capacity}
			Given $\gamma\in(0,1]$, there exists $\delta=\delta(N,\gamma)>0$
			and a modulus of continuity $\omega$ such that for every nearly spherical set $\Omega$
			parametrized by $\varphi$ with $\Vert\varphi\Vert_{C^{2,\gamma}(\partial B_1)}<\delta$
			and $\vert\Omega\vert=\vert B_1\vert$, we have
			\begin{equation*}
				\Capa_p(\Omega)\geq \Capa_p(B_1)+\frac{1}{2}\partial^2 \Capa_p(B_1)[\varphi,\varphi]
				-\omega(\Vert\varphi\Vert_{C^{2,\gamma}})\Vert\varphi\Vert^2_{H^\frac{1}{2}(\partial B_1)},
			\end{equation*}
			where
			\begin{equation*}
			\begin{split}
				\left(\frac{p-1}{N-p}\right)^{p-2}\frac{1}{p}\partial^2 \Capa_p(B_1)[\varphi,\varphi]:=				
				-(N-1)\int_{\partial B}{\left(\frac{N-p}{p-1}\varphi\right)^2}\,d\mathcal{H}^{N-1}\\
				+\int_{B^c}{\vert x\vert^{(p-2)\left(\frac{p-N}{p-1}-1\right)}\left(\vert\nabla f(\varphi)\vert^2+(p-2)\left(\theta\cdot\nabla f(\varphi)\right)^2\right)}\,dx,
			\end{split}	
			\end{equation*}
			with $f(\varphi)$ satisfying
			\begin{equation*}
			\begin{cases}
				\div\left(\vert x\vert^{(p-2)(\frac{p-N}{p-1}-1)}\nabla f(\varphi)+(p-2)\vert x\vert^{(p-2)(\frac{p-N}{p-1}-1)}(\theta\cdot\nabla f(\varphi))\theta\right)=0\text{ in }B^c,\\
				f(\varphi)=\frac{N-p}{p-1}\varphi\text{ on }\partial B.
			\end{cases}
			\end{equation*}
		\end{lemma}		
		\begin{proof}
		We write Taylor expansion for $c_\kappa$:
		$$c_\kappa(1)=c_\kappa(0)+c_\kappa'(0)+\frac{1}{2}c_\kappa''(0)+\frac{1}{2}\int_0^1{(1-t)(c_\kappa''(t)-c_\kappa''(0))}\,dt.$$
		From isocapacitary inequality we know that $c_0'(0)=0$.
		So, we get the desired inequality using Lemma \ref{l:continuity of the second derivative p cap} and passing to the limit as $\kappa\rightarrow 0$. 
		\end{proof}
		
\subsection{Inequality for nearly spherical sets}\label{sec:lin}
We now establish a quantitative inequality for nearly spherical sets in the spirit of those established by Fuglede in~\cite{Fuglede89}, compare with \cite[Section 2]{FZ}.

\begin{theorem} \label{mainthmnrlysphr}
	There exists $\delta=\delta(N), c=c(N)$ such that if $\Omega$ is a nearly spherical set of class $C^{2,\gamma}$
	parametrized by $\varphi$ with $\Vert\varphi\Vert_{C^{2,\gamma}}\leq\delta,\vert\Omega\vert=\vert B_1\vert$
	and $x_\Omega=0$, then
	\begin{equation*}
		\Capa_p(\Omega)-\Capa_p(B_1)\geq c\Vert\varphi\Vert^2_{H^\frac{1}{2}(\partial B_1)}.
	\end{equation*}
\end{theorem}
\begin{rem}
	Note that by Lemma \ref{propasym} this theorem gives us Theorem \ref{thm:main p-cap} for nearly spherical sets.
\end{rem}
\begin{proof}
			Let $X$ be the vector field from Lemma \ref{lm:vectorfield}.
			We introduce the following notation:
			\begin{equation*}
				\hat{u}=-\frac{p-1}{p-N}\dot{u}_0,\text{ }\Psi=\theta\cdot X.
			\end{equation*}
			Then by Proposition \ref{prop:second derivative p-capacity} $\hat{u}$ solves
			\begin{equation*}
			\begin{cases}
				\div\left(\vert x\vert^{(p-2)(\frac{p-N}{p-1}-1)}\nabla \hat{u}+(p-2)\vert x\vert^{(p-2)(\frac{p-N}{p-1}-1)}(\theta\cdot\nabla\hat{u})\theta\right)=0\text{ in }B^c,\\
				\hat{u}=\Psi\text{ on }\partial B
			\end{cases}
			\end{equation*}
			and
			\begin{equation*}
			\begin{split}
				\left(\frac{p-1}{p-N}\right)^{p}\frac{1}{p}c''_0(0)&=-(N-1)\int_{\partial B}{\hat{u}^2}d\mathcal{H}^{N-1}\\
				&+\int_{B^c}{\vert x\vert^{(p-2)\left(\frac{p-N}{p-1}-1\right)}\left(\vert\nabla\hat{u}\vert^2+(p-2)(\theta\cdot\nabla\hat{u})^2\right)}dx.
			\end{split}	
			\end{equation*}
			We introduce the following notation:
			\begin{equation*}
			\begin{split}
			Q[\Psi]&=-(N-1)\int_{\partial B}{\hat{u}^2}d\mathcal{H}^{N-1}\\
				&+\int_{B^c}{\vert x\vert^{(p-2)\left(\frac{p-N}{p-1}-1\right)}\left(\vert\nabla\hat{u}\vert^2+(p-2)(\theta\cdot\nabla\hat{u})^2\right)}dx.
			\end{split}
			\end{equation*}
			We write $\Psi$ in the basis of spherical harmonics, i.e.
			\begin{equation*}
				\Psi=\sum_{k=0}^\infty\sum_{i=1}^{M(k,N)}a_{k,i}Y_{k,i},
			\end{equation*}
where $Y_{k,i}$ for $i=1,\dots, M(k,N)$ are harmonic polynomials of degree $k$, normalized so that $\Vert Y_{k,i}\Vert_{L^2(\partial B)}=1$. 
			By Lemma \ref{lm:vectorfield} we have that $\Vert\Psi\Vert_{H^{1/2}(\partial B_1)}\geq c\Vert\varphi\Vert_{H^{1/2}(\partial B_1)}$ if $\delta$ is small enough. Thus, to prove the theorem, by Lemma \ref{l:taylor for p-capacity} it is enough to show that $Q[\Psi]\geq c \Vert\Psi\Vert_{H^{1/2}(\partial B_1)}^2$, or, equivalently, that
		  \begin{equation}\label{e:ineq for Q}
		  Q[\Psi]=\sum_{k=0}^\infty\sum_{i=1}^{M(k,N)}a_{k,i}^2 Q[Y_{k,i}]\geq c\sum_{k=0}^\infty\sum_{i=1}^{M(k,N)}(k+1)a_{k,i}^2.
		  \end{equation}
		  
		\medskip
		\noindent
		We first note $\int_{\partial B}\Psi=0$ as $\Phi_t$ conserves volume, and thus $a_0=0$.		  
		We then bound $\sum_{i=1}^n a_{1,i}^2$. We recall that $x_\Omega=0$, hence
		$$
		\int_{\partial B}{x\left((1+\varphi)^{N+1}-1\right)}\,d\mathcal{H}^{N-1}=0
		$$
		and consequently, for any $\varepsilon>0$ if $\delta$ is small enough we get
		$$
		\left\vert\int_{\partial B}{x\varphi}\,d\mathcal{H}^{N-1}\right\vert\leq\varepsilon\Vert\varphi\Vert_{L^2(\partial B)}.
		$$
		By Lemma \ref{lm:vectorfield} this in turn yields
		$$
		\left\vert\int_{\partial B}{x\Psi}\,d\mathcal{H}^{N-1}\right\vert\leq 2\varepsilon\Vert\Psi\Vert_{L^2(\partial B)},
		$$
		if $\delta$ is small enough.
		So we get $$\sum_{i=1}^n a_{1,i}^2\leq 2\sum_{k=2}^\infty\sum_{i=1}^{M(k,N)}(k+1)a_{k,i}^2$$
		for $\delta$ small enough and to prove \eqref{e:ineq for Q}
		it remains to show that
		  \begin{equation}\label{e:ineq for Q reduced}
		  \sum_{k=0}^\infty\sum_{i=1}^{M(k,N)}a_{k,i}^2 Q[Y_{k,i}]\geq c\sum_{k=2}^\infty\sum_{i=1}^{M(k,N)}(k+1)a_{k,i}^2.
		  \end{equation}
We denote by $u_{k,i}$ the function $\hat{u}$ corresponding to $Y_{k,i}$ on the boundary.		
			Then a straightforward computation tells us that 
			$$u_{k,i}=\vert x\vert^{\alpha_k}Y_{k,i},$$
			where $\alpha_k<0$ is the only negative solution of the following quadratic equation:
			\begin{equation}\label{e:eq for alphha}
				(p-1)\alpha_k^2+\left(N-p\right)\alpha_k-k(k+N-2)=0.
			\end{equation}
			Remembering that
			$$\int_{\partial B}\vert Y_{k,i}\vert^2d\mathcal{H}^{N-1}=1,
			\int_{\partial B}\vert\nabla_\tau Y_{k,i}\vert^2d\mathcal{H}^{N-1}=k(k+N-2),$$
			we get that
			\begin{equation*}
			\begin{split}
				Q[Y_{k,i}]=-(N-1)-\frac{\alpha_k^2(p-1)+k(k+N-2)}{(p-2)(\frac{p-N}{p-1}-1)+2(\alpha_k-1)+N}\\
				=-(N-1)-\frac{\left(2k(k+N-2)-(N-p)\alpha_k\right)(p-1)}{(p-1)2\alpha_k+N-p},
			\end{split}	
			\end{equation*}
			where we used \eqref{e:eq for alphha}.
			Now, since by \eqref{e:eq for alphha} we have
			\begin{equation*}
				\alpha_k=-\frac{N-p+\sqrt{(N-p)^2+4(p-1)k(k+N-2)}}{2(p-1)},
			\end{equation*}
			we get after straightforward computations
			\begin{equation*}
			\begin{split}
				Q[Y_{k,i}]=-(N-1)+\frac{N-p+\sqrt{(N-p)^2+4(p-1)k(k+N-2)}}{2}.
			\end{split}	
			\end{equation*}
			Notice that 
			\begin{equation*}
			Q[Y_{k,i}]\geq ck\text{ for }k\geq 2
			\end{equation*} 
			for some $c=c(N,p)>0$. This gives us \eqref{e:ineq for Q reduced} and hence \eqref{e:ineq for Q} and so we conclude the proof of the theorem.	
\end{proof}

	\section{Stability  for bounded sets with small asymmetry}\label{penpb}

	This section, together with the two subsequent ones, is dedicated to the proof of the following theorem.	
	
	\begin{theorem} \label{pcapmainthmbdd}
		There exist constants $c=c(N,R)$, $\varepsilon_0=\varepsilon_0(N,R)$ such that
		for any open set $\Omega\subset B_R$ with $\vert\Omega\vert=\vert B_1\vert$ and $\alpha(\Omega)\leq\varepsilon_0$
		the following inequality holds:
		\begin{equation*}
			\Capa_p(\Omega)-\Capa_p(B_1)\geq c\alpha(\Omega).
		\end{equation*}
	\end{theorem}
	
	To prove Theorem \ref{pcapmainthmbdd} we are going to argue by contradiction. Suppose that the theorem doesn't hold. Then there exists a sequence of open sets $\tilde{\Omega}_j\subset B_R$ such that
	\begin{equation}\label{pcapcontrseq}
		\vert \tilde{\Omega}_j\vert=\vert B_1\vert,\quad
		\alpha(\tilde{\Omega}_j)=\varepsilon_j\rightarrow 0,\quad
		\frac{\Capa_p(\tilde{\Omega}_j)-\Capa_p(B_1)}{\varepsilon_j}\leq\sigma^4
	\end{equation}
	for some small $\sigma$ to be chosen later.
	We then perturb the sequence $\tilde{\Omega}_j$ so that it converges to $B_1$ in a smooth way. More precisely, we are going to show the following.		
	
	\begin{theorem}[Selection Principle]\label{pcapSelPr}
		There exists $\tilde{\sigma}=\tilde{\sigma}(N,R)$ such that 
		if one has a contradicting sequence $\tilde{\Omega}_j$ as 
		the one described above in \eqref{pcapcontrseq} with $\sigma<\tilde{\sigma}$, 
	          then there exists a sequence of smooth open sets $U_j$ such that
		\begin{enumerate}[label=\textup{(\roman*)}]
			\item $\vert U_j\vert=\vert B_1\vert$,
			\item $\partial U_j\rightarrow \partial B_1$ in $C^k$ for every $k$,
			\item $\limsup_{j\rightarrow\infty}\frac{\Capa_p(U_j)-\Capa_p(B_1)}{\alpha(\Omega_j)}\leq C\sigma$ for some $C=C(N,R)$ constant,
			\item the barycenter of every $\Omega_j$ is in the origin.
		\end{enumerate}
	\end{theorem}
	
	\begin{proof} [Proof of Theorem \ref{pcapmainthmbdd} assuming Selection Principle]
		Suppose Theorem \ref{pcapmainthmbdd} does not hold. Then for any $\sigma>0$ we can find a contradicting
		sequence $\tilde{\Omega}_j$ as in \eqref{pcapcontrseq}. We apply Theorem \ref{pcapSelPr} to $\tilde{\Omega}_j$
		to get a smooth contradicting sequence $U_j$.
		
		By the properties of $U_j$, we have that for $j$ big enough $U_j$ is a nearly spherical set with barycenter at the origin
		and with volume $\vert B_1\vert$.
		Thus, we can use Theorem \ref{mainthmnrlysphr} and get
		$$c(N,R)\leq\limsup_{j\rightarrow\infty}\frac{\Capa_*(U_j)-\Capa_*(B_1)}{\alpha_*(\Omega_j)}\leq C(N,R)\sigma.$$
		But this cannot happen for $\sigma=\sigma(N,R)$ small enough.
	\end{proof}

\section{Proof of Theorem \ref{pcapSelPr}: Existence and first properties}	\label{s:ex}
	
\subsection{Getting rid of the volume constraint}
The first step consists in getting rid of the volume constraint in the isocapacitary inequality. Note that this has to be done locally  since, by scaling,  globally there exists no Lagrange multiplier. Furthermore, to apply the regularity theory for free boundary problems, it is crucial to introduce a \emph{monotone} dependence on the volume. To this end, let us set, following \cite{AguileraAltCaffarelli86},
	\begin{equation*}
		f_\eta(s):=
		\begin{cases}
			-\frac{1}{\eta}(s-\omega_N), \qquad &s\leq\omega_N\\
			-\eta(s-\omega_N), &s\geq\omega_N
		\end{cases}
	\end{equation*} 
	and let us consider the new functional	
	\begin{equation*}
		\mathscr{C}_\eta(\Omega)=\Capa_p(\Omega)+f_\eta(\vert\Omega\vert).
	\end{equation*}
We now show that the above functional is uniquely minimized by balls. Note also that \(f_\eta\) satisfies
\begin{equation}
\label{e:feta}
\eta(t-s)\le f_\eta(s)-f_\eta(t)\le \frac{(t-s)}{\eta}\qquad\text{for all \(0\le s\le t\).}
\end{equation}

	\begin{lemma} \label{pcappropofCeta}
		There exists an $\hat{\eta}=\hat{\eta}(R)>0$ such that the only minimizer of 
		$\mathscr{C}_{\hat{\eta}}$ in the class of sets contained in $B_{R}$
		is  a translate of the unit  unit ball $B_1$.
		
		Moreover, there exists $c=c(R)>0$ such that for any ball $B_r$ with $0<r<R$, one has
		\begin{equation*}
			\mathscr{C}_{\hat{\eta}}(B_r)-\mathscr{C}_{\hat{\eta}}(B_1)\geq c\vert r-1\vert.
		\end{equation*}
	\end{lemma}
	\begin{proof}
		Suppose that $\Omega$ is a minimizer of 
		$\mathscr{C}_{\hat{\eta}}$ in the class of sets contained in $B_{R}$. Let $r$ be such that $\vert B_r\vert=\vert\Omega\vert$.
		Then, by symmetrization 
		\begin{equation*}
			\mathscr{C}_{\hat{\eta}}(\Omega)=\Capa_p(\Omega)+f_{\hat{\eta}}(\vert\Omega\vert)\geq\Capa_p(B_r)+f_{\hat{\eta}}(\vert\Omega\vert)
			=\mathscr{C}_{\hat{\eta}}(B_r).
		\end{equation*}
		Moreover, the equality holds only if $\Omega$ is a translation
		of the ball $B_r$. So, any minimizer should be a translate of 
		a ball of a radius $r\in(0, R]$, and we need to prove that
		a function $g:(0,R]\rightarrow\mathds{R}_+$ defined as
		\begin{equation*}
			g(r):=\Capa_p(B_r)+f_{\hat{\eta}}(\vert B_r\vert)
			=\Capa(B_1)r^{n-p}+f_{\hat{\eta}}(\vert B_r\vert)
		\end{equation*}		 
		achieves its only minimum at $r=1$ if $\hat{\eta}$
		is small enough. This is a problem on a real line and
		the desired result can be obtained by examining
		the derivative of $g$ in separate cases $r\in(0,1)$ and $r\in[1,R]$.
		
		For more detail look at the proofs of Lemma 4.1 and Lemma 4.2 in \cite{DPMM}. The computation can be repeated almost verbatim, the only difference being scaling of capacity.
	\end{proof}	

	\subsection{A penalized minimum problem}
	The sequence in Theorem \ref{pcapSelPr} is obtained by solving the following minimum problem (see also \cite{AFM} where
	a similar penalized problem is introduced to deal with a nonlocal isoperimetric problem).
	\begin{equation} \label{pcappertpb}
		\min{\{\mathscr{C}_{\hat{\eta},j}(\Omega):\Omega\subset B_R\}},
	\end{equation}
	where 
	\[
	\mathscr{C}_{\hat{\eta},j}(\Omega)=\mathscr{C}_{\hat{\eta}}(\Omega)+\sqrt{\varepsilon_j^2+\sigma^2(\alpha(\Omega)-\varepsilon_j)^2}
	=\Capa_p(\Omega)+f_{\hat{\eta}}(\vert\Omega\vert)+\sqrt{\varepsilon_j^2+\sigma^2(\alpha(\Omega)-\varepsilon_j)^2}.
	\]

	\begin{lemma}
		There exists $\sigma_0=\sigma_0(N,R)>0$ such that for every $\sigma<\sigma_0$
		the minimum in \eqref{pcappertpb} is attained by a $p$-quasi-open\footnote{Recall that a set is said to be $p$-quasi-open if it is the zero level set of a $W^{1,p}$ function.} set $\Omega_j$.
		Moreover, perimeters of $\Omega_j$ are bounded independently on $j$.
	\end{lemma}
		\begin{proof}
		The proof is almost verbatim repetition of the proof of 
		\cite[Lemma 4.3]{DPMM}. 		
			
		\noindent	
		\textbf{Step 1: finding minimizing sequence with  bounded perimters.}		
		We consider $\{V_k\}_{k\in\mathbb{N}}$ -- a minimizing sequence for $\mathscr{C}_{\hat{\eta},j}$, satisfying
		\begin{equation*}
			\mathscr{C}_{\hat{\eta},j}(V_k)\leq \inf{\mathscr{C}_{\hat{\eta},j}}+\frac{1}{k}.
		\end{equation*}
		We denote by $v_k$ the capacitary potentials of $V_k$, so $V_k=\{x\in B_R: v_k=1\}$.
		We take as a variation the slightly enlarged set $\tilde{V}_k$:
		\begin{equation*}
			\tilde{V}_k=\{x\in B_R: v_k>1-t_k\},
		\end{equation*}
		where $t_k=\frac{1}{\sqrt{k}}$.
		
		Note that the function $\tilde{v}_k=\frac{\min(v_k, 1-t_k)}{1-t_k}$ is in $D^{1,p}(\mathbb R^N)$ and  \(v_k=1\) on $\tilde{V}_k$, so we can bound the p-capacity of $\tilde{V}_k$ by $\int{\vert\nabla\tilde{v}_k\vert^p}dx$. Since $V_k$ is almost minimizing, we write
		\begin{equation*}
			\begin{aligned}
				&\int_{\{v_k<1\}}{\vert\nabla v_k\vert^p}dx+f_{\hat{\eta}}(|\{v_k=1\}|)+\sqrt{\varepsilon_j^2+\sigma^2(\alpha(\{v_k=1\})-\epsilon_j)^2}\\
				&\leq\int_{\{v_k<1-t_k\}}{\left\vert\nabla \left(\frac{v_k}{1-t_k}\right)\right\vert^p}dx+f_{\hat{\eta}}(|\{v_k\geq 1-t_k\}|)+\sqrt{\varepsilon_j^2+\sigma^2(\alpha(\{v_k\geq 1-t_k\})-\varepsilon_j)^2}+\frac{1}{k}.
			\end{aligned}	
		\end{equation*}
		We use \eqref{e:feta} and the fact that the function $t\mapsto\sqrt{\varepsilon_j^2+\sigma^2(t-\varepsilon_j)^2}$ is \(1\) Lipschitz to get
		\begin{equation*}
			\begin{aligned}
				&\int_{\{1-t_k<v_k<1\}}{\vert\nabla v_k\vert^p}dx+\hat\eta|\{1-t_k<v_k<1\}|\\
				&\leq\sigma\big(\vert\alpha(\{v_k\geq 1-t_k\}-\alpha(\{v_k=1\})\vert\big)+\frac{1}{k}+\int_{\{v_k<1-t_k\}}{\left(\left(\frac{1}{1-t_k}\right)^p-1\right)\vert\nabla v_k\vert^p}dx\\
				&\leq	C(R)\sigma\vert\{1-t_k<v_k\leq 1\}\vert+\frac{1}{k}+\left(\left(\frac{1}{1-t_k}\right)^p-1\right) \Capa_p(V_k)\\
				&\leq	C(R)\sigma\vert\{1-t_k<v_k\leq 1\}\vert+\frac{1}{k}+c(N,R)t_k,
			\end{aligned}	
		\end{equation*}
		where in the second inequality we used Lemma \ref{propasym}, \ref{asymlip}. Taking $\sigma<\frac{\hat\eta}{2C(R)}$, we obtain
		\begin{equation*}
			\int_{\{1-t_k<v_k<1\}}{\vert\nabla v_k\vert^p}dx+\frac{\hat \eta}{2}(|\{1-t_k<v_k<1\}|)\leq\frac{1}{k}+c(N,R)t_k.
		\end{equation*}
		We estimate the left-hand side from below, using the arithmetic-geometric  mean inequality, the Cauchy-Schwarz inequality, and the co-area formula.
		\begin{equation*}
			\begin{aligned}
				&\int_{\{1-t_k<v_k<1\}}{\vert\nabla v_k\vert^p}dx+\frac{\hat\eta}{2}(|\{1-t_k<v_k<1\}|)\\
				&\geq p\left(\int_{1-t_k<v_k<1}{\vert\nabla v_k\vert^p}dx\right)^\frac{1}{p}\left((p-1)\frac{\hat \eta}{2}(|\{1-t_k<v_k<1\}|)\right)^\frac{1}{p'}\\
				&\geq c{\hat\eta}^\frac{1}{p'}\int_{1-t_k<v_k<1}{\vert\nabla v_k\vert}dx
				=c{\hat\eta}^\frac{1}{p'}\int_{1-t_k}^1{P(\{v_k>s\})}ds.
			\end{aligned}	
		\end{equation*}
where \(P(E)\) denotes the De Giorgi perimeter of a set \(E\). Hence, there exists a level $1-t_k<s_k<1$ such that for $\hat{V}_k=\{v_k>s_k\}$
		\begin{equation*}
			P(\hat{V}_k)\leq\frac{1}{t_k}\int_{1-t_k}^1{P(\{v_k>s\})}ds\leq\frac{1}{t_kc{\hat\eta}^\frac{1}{p'}k}+c(N,R)=\frac{1}{c{\hat\eta}^\frac{1}{p'}\sqrt{k}}+c(N,R).
		\end{equation*}	
		where in the last equality we have used that \(t_k=\frac{1}{\sqrt{k}}\).  
		These $\hat{V}_k$ will give us the desired "good" minimizing sequence, indeed 
		\begin{equation*}
			\begin{split}
				\mathscr{C}_{\hat{\eta},j}(\hat{V}_k)				
				\leq\mathscr{C}_{\hat{\eta},j}(V_k)+f_{\hat{\eta}}(\vert\{ v_k>s_k\}\vert)
				&-f_{\hat{\eta}}(\vert \{v_k=1\}\vert)\\
				&+C\sigma\vert\{1-s_k<v_k<1\}\vert \leq\mathscr{C}_{\hat{\eta},j}(V_k),
			\end{split}
		\end{equation*}
		where in the first inequality we have used that \(\hat V_k\subset V_k\) and in the second that, thanks to our choice of \(\sigma\),
		\[
		f_{\hat{\eta}}(\vert \{v_k>s_k\}\vert)-f_{\hat{\eta}}(\vert \{v_k=1\}\vert)
				+C\sigma\vert\{1-s_k<v_k<1\}\le (C\sigma-\hat\eta)\vert\{1-s_k<v_k<1\}\le 0.
		\]
		
\medskip
\noindent
\textbf{Step 2: Existence of a minimizer.}		
		Since $\{\hat{V}_k\}_k$ is a sequence with  equibounded perimeter,s there exists a Borel set $\hat{V}_\infty$ such that up to a (not relabelled) subsequence 
		$$ 1_{\hat{V}_k}\rightarrow 1_{\hat{V}_\infty} \text{ in } L_1(B_R)\text{ and a.e. in }B_R, \qquad P(\hat{V}_\infty)\leq C(N,R).$$
We want to show that $\hat{V}_\infty$ is a minimizer for $\mathcal{C}_{\eta,j}$.	We set $\hat{v}_k=\frac{\min(v_k, s_k)}{s_k}$ and we note that they are the capacitary potentials of \(\hat V_k\). Moreover  the sequence $\{\hat{v}_k\}_k$ is bounded in $D^{1,p}(\mathbb R^N)$. Thus, there exists a function $\hat{v}\in D^{1,p}(\mathbb R^N)$ such that up to a (not relabelled) subsequence $$\hat{v}_k\rightarrow \hat{v} \text{ strongly in }L^p(B_R)\text{ and a.e. in }B_R.$$
		Let us define  $\hat{V}=\{x:\hat{v}=1\}$, we want to show that \(\hat V\) is a minimizer. First,  note that
		$$1_{\hat{V}}(x)\geq \limsup{1_{\hat{V}_k}}(x)=1_{\hat{V}_\infty}(x)\qquad\text{ for a.e. }x\in B_R,$$
hence \(|\hat V_\infty \setminus \hat V|=0\). Moreover, by the lower semicontinuity of Dirichlet integral, the monotonicity of \(f_{\hat{\eta}}\) and the continuity of $\alpha$ with respect to the $L^1$ convergence, we have
		\begin{equation}
		\begin{aligned} \label{pcapVinftyismin}
		\inf\mathscr{C}_{\hat{\eta},j}&=\lim_k \int \vert\nabla\hat{v}_k\vert^p+f_{\hat{\eta}}(\vert\hat{V}_k\vert)+\sqrt{\varepsilon_j^2+\sigma^2(\alpha(\hat{V}_k)-\varepsilon_j)^2}
			\\
			&\ge \Capa_p(\hat{V})+f_{\hat{\eta}}(\vert\hat{V}_\infty\vert)
			+\sqrt{\varepsilon_j^2+\sigma^2(\alpha(\hat{V}_\infty)-\varepsilon_j)^2}
			\ge \Capa_p(\hat{V})+f_{\hat{\eta}}(\vert\hat{V}\vert).
		\end{aligned}	
		\end{equation}
		Hence
		\begin{equation*}
		\begin{aligned} 
			\Capa_p(\hat{V})&+f_{\hat{\eta}}(\vert\hat{V}_\infty\vert)
			+\sqrt{\varepsilon_j^2+\sigma^2(\alpha(\hat{V}_\infty)-\varepsilon_j)^2}
			\leq\inf\mathscr{C}^R_{\hat{\eta},j}(\Omega)\\
			&\leq \Capa_p(\hat{V})+f_{\hat{\eta}}(\vert\hat{V}\vert)
			+\sqrt{\varepsilon_j^2+\sigma^2(\alpha(\hat{V})-\varepsilon_j)^2}.
		\end{aligned}	
		\end{equation*}
		Using Lemma \ref{propasym} \ref{asymlip}   we get 
		$$f_{\hat{\eta}}(\vert\hat{V}_\infty\vert)-f_{\hat{\eta}}(\vert\hat{V}\vert)\leq C\sigma\vert\hat{V}\Delta\hat{V}_\infty\vert=C\sigma\vert\hat{V}\setminus \hat{V}_\infty\vert.$$
		Since $\vert\hat{V}\vert\geq\vert\hat{V}_\infty\vert$, \eqref{e:feta} and  our choice of \(\sigma\)   yield
		$$\hat \eta\vert\hat{V}\backslash\hat{V}_\infty\vert\leq f_{\hat{\eta}}(\vert\hat{V}_\infty\vert)-f_{\hat{\eta}}(\vert\hat{V}\vert)\leq C\sigma\vert\hat{V}\backslash\hat{V}_\infty\vert\le \frac{\hat{\eta}}{2}\vert\hat{V}\backslash\hat{V}_\infty\vert,$$
	from which we conclude that \(|\hat V\Delta \hat V_\infty|=0\) and thus, by \eqref{pcapVinftyismin} that \(\hat V\) is the desired minimizer.
	\end{proof}

	\subsection{First properties of the minimizers}
	
	\begin{lemma} \label{pcapfirstpropmin}
		Let $\{\Omega_j\}$ be a sequence of minimizers for \eqref{pcappertpb}.
		Then the following properties hold:
		\begin{enumerate}[label=\textup{(\roman*)}]
			\item 
			$\vert\alpha(\Omega_j)-\varepsilon_j\vert\leq 3\sigma\varepsilon_j$;
			\item 
			$\big\vert\vert\Omega_j\vert-\vert B_1\vert\big\vert\leq C\sigma^4\varepsilon_j$;
			\item\label{pcap L1 convergence to the ball}
						 up to translations $\Omega_j\rightarrow B_1$ in $L^1$,
			\item 
			$0\leq\mathscr{C}_{\hat{\eta}}(\Omega_j)-\mathscr{C}_{\hat{\eta}}(B_1)\leq\sigma^4\varepsilon_j$.
		\end{enumerate}
	\end{lemma}
	\begin{proof}
	The lemma follows easily from Lemma \ref{pcappropofCeta}.
	To prove \ref{pcap L1 convergence to the ball} we need
	to recall that the sets $\Omega_j$ have bounded perimeter.
	
	For more details see the proof of \cite[Lemma 4.4]{DPMM}.
	\end{proof}
	\section{Proof of Theorem \ref{pcapSelPr}: Regularity}\label{reg}
	
In this section, we show that  the sequence of minimizers of  \eqref{pcappertpb} converges smoothly to  the unit ball. This will be done by relying on the regularity theory for free boundary problems established in \cite{DP}.
	
	\subsection{Linear growth away from the free boundary} Let $u_j$ be the capacitary potential for $\Omega_j$, a minimizer of \eqref{pcappertpb}.	Let us also introduce $v_j:=1-u_j$, so that $\Omega_j=\{v_j=0\}$. Following \cite{DP} we are going to show that 
	\[
	v_j(x)\sim \dist(x,\Omega_j).
	\]
where the implicit constant depends only on \(R\). The above estimate is obtained by suitable comparison estimates. We will need to have
some compactness properties, so we first prove H\"older continuity,
also with the constant depending only on \(R\).
	
	\subsubsection{H\"older continuity} 
	
	The proof is based on establishing a decay estimate for the integral oscillation of $u_j$ 
	and it is almost identical to the case of $2$-capacity
	(see \cite[Lemma 5.8]{DPMM}).

	We are going to use the following growth result for $p$-harmonic
	functions. The proof can be found, for example, in \cite[Theorem 7.7]{G}.	
	
	\begin{lemma} \label{l:pharmgrowth}
		Suppose $w\in W^{1,2}(\Omega)$ is $p$-harmonic, $x_0\in\Omega$.
		Then there exists a constant $c=c(N,p)$, $0<\beta\leq 1$ such that for any balls $B_{r_1}(x_0)\subset B_{r_2}(x_0)\Subset\Omega$	
		\begin{equation*}
			\fint_{B_{r_1}(x_0)}{\left\vert \nabla w\right\vert^p}\leq c\left(\frac{r_1}{r_2}\right)^{p\beta-p}\fint_{B_{r_2}(x_0)}{\left\vert\nabla w\right\vert^p}.
		\end{equation*}
	\end{lemma}
	\begin{rem}
		In \cite{G} the result is proven for the functions in De Giorgi class.
		One can prove that in the case of $p$-harmonic functions the
		inequality holds for $\beta=1$, but we are not going to need that.
	\end{rem}
		
	To prove H\"{o}lder continuity of $u_j$  we will use  several times the following comparison estimates. 
	\begin{lemma}\label{l:varpcap}Let \(u_j\) be the capacitary potential of  a minimizer for \eqref{pcappertpb}. Let $A\subset B_R$ be an open set with Lipschitz  boundary and  let $w\in W^{1,p}(\mathds{R}^n)$ coincide with $u_j$ on the boundary of $A$ in the sense of traces.

		Then
		$$\int_A{\vert\nabla u_j\vert^p}dx-\int_A{\vert\nabla w\vert^p}dx\leq \left(\frac{1}{\hat \eta}+C\sigma\right)\big| A\cap \left(\{u=1\}\Delta\{w=1\}\right)\big|.$$
		Moreover, if $u_j\leq w\leq 1$ in $A$, then
		$$\int_A{\vert\nabla u_j\vert^p}dx+\frac{\hat \eta}{2}\big|A\cap \left(\{u=1\}\Delta\{w=1\}\right)\big|\leq\int_A{\vert\nabla w\vert^p}dx,$$
		provided \(\sigma\le \sigma(N,R)\).
	\end{lemma}
	\begin{proof}
	The proof is the same as the proof of \cite[Lemma 5.5]{DPMM},
	modulo changing exponents from $2$ to $p$. The idea is to consider
	$\tilde{u}$ defined as
		$$
			\begin{cases}
				\tilde{u}=w\qquad&\text{ in }A\\
				\tilde{u}=u&\text{ else}
			\end{cases}
		$$
	and take $\tilde{\Omega}=\{\tilde{u}=1\}$ as a comparison domain.	
	\end{proof}
	\begin{rem}\label{varcapharm}
		Note that if $w$ is $p$-harmonic in $A$, then by Lemma \ref{l:ineq for difference of powers} 
 		\begin{itemize}
 			\item if $p\geq 2$, we have
 			\begin{equation*}
 				\cpp{A}{u}-\cpp{A}{w}\geq c\cpp{A}{(u-w)}; 
 			\end{equation*}
 			\item if $1<p<2$, then
 			\begin{equation*}
 				\cpp{A}{u}-\cpp{A}{w}\geq c\int_A{\vert\nabla(u-w)\vert^2\left(\vert\nabla w\vert^2+\vert \nabla(u-w)\vert^2\right)^{\frac{p-2}{2}}}\,dx. 
 			\end{equation*}
 		\end{itemize}
		Hence the first inequality from the lemma becomes
		\begin{itemize}
		\item for $p\geq 2$
		\begin{equation}\label{e:pharmonic_p bigger than 2}
		\cpp{A}{(u-w)}\leq C(p)\left(\frac{1}{\hat \eta}+C\sigma\right)\big| A\cap\left(\{u=1\}\Delta\{w=1\}\right)\big|;
		\end{equation}
		\item for $1<p<2$ 
		\begin{equation}\label{e:pharmonic_p less than 2}
		\begin{split}
		&\int_A{\vert\nabla(u-w)\vert^2\left(\vert\nabla w\vert^2+\vert \nabla(u-w)\vert^2\right)^{\frac{p-2}{2}}}\,dx\\
		&\qquad\leq C(p)\left(\frac{1}{\hat \eta}+C\sigma\right)\big| A\cap\left(\{u=1\}\Delta\{w=1\}\right)\big|.
		\end{split}
		\end{equation}
		\end{itemize}
	\end{rem}

Let us also recall the following technical result.
	\begin{lemma}[Lemma 5.13 in \cite{GM}] \label{techlemmagrowth}
		Let $\phi:\mathds{R}^+\rightarrow\mathds{R}^+$ be a non-decreasing function satisfying
		$$\phi(\rho)\leq A\left[\left(\frac{\rho}{R}\right)^\alpha+\varepsilon\right]\phi(R)+BR^\beta,$$
		for some $A,\alpha,\beta>0$, with $\alpha>\beta$ and for all $0<\rho\leq R\leq R_0$, where $R_0>0$ is given. 
		Then there exist constants $\varepsilon_0=\varepsilon_0(A,\alpha,\beta)$ and $c=c(A,\alpha,\beta)$ 
		such that if $\varepsilon\leq\varepsilon_0$, we have
		$$\phi(\rho)\leq c\left[\frac{\phi(R)}{R^\beta}+B\right]\rho^\beta$$
		for all $0\leq\rho\leq R\leq R_0$.
	\end{lemma}	
		
	We are now ready to prove H\"older continuity of $u_j$.	
	
	\begin{lemma} \label{l:pcapholder_continuity} There exists \(\alpha\in (0,1/2)\) such that the capacitary potential of every minimizer of \eqref{pcappertpb} satisfies 
		$u_j\in C^{0,\alpha}(\overline{B_R})$. Moreover, the H\"older norm is bounded by a constant independent  on $j$.
	\end{lemma}
	\begin{proof}

	The proof is similar to the proof of \cite[Lemma 5.8]{DPMM}.
		As usual, we drop the subscript $j$.
	By Morrey Theorem (see, for example, \cite[Theorem 5.7]{GM}) it is enough to show that 
\[
\phi(r):=\int_{B_r(x_0)}{\vert\nabla u\vert^p}\le Cr^{N+2\alpha-p}
\]		
for all \(r\) small enough (say less that \(1/2\)).
	
	\noindent
	Let $x_0\in B_R$.	Let  $w$ be the p-harmonic extension of \(u\) 	in $B_{r'}(x_0)$. 
	By Lemma \ref{l:pharmgrowth} we know that 
		\begin{equation*}
			\int_{B_r(x_0)}{\left\vert\nabla w\right\vert^p}\leq C\left(\frac{r}{r'}\right)^{N+p\beta-p}\int_{B_{r'}(x_0)}{\left\vert\nabla w\right\vert^p}.
		\end{equation*} 
		Let $g:=u-w$. Then
		\begin{equation*}
			\begin{aligned}
				&\int_{B_r(x_0)}{\left\vert \nabla u\right\vert^p}dx
				\leq C\int_{B_r(x_0)}{\left\vert\nabla w\right\vert^p}dx+C\int_{B_r(x_0)}{\left\vert \nabla g\right\vert^p}\\
				&\qquad\leq C\left(\frac{r}{r'}\right)^{N+p\beta-p}\int_{B_{r'}(x_0)}{\left\vert\nabla w\right\vert^p}dx+C\int_{B_{r'}(x_0)}{\left\vert \nabla g\right\vert^p}.
			\end{aligned}	
		\end{equation*}
		We want to show the following bound:
		\begin{equation}\label{e:bound on dirichlet energy of g}
			\int_{B_{r'}(x_0)}{\left\vert \nabla g\right\vert^p}
			\leq C_\varepsilon(r')^N+C\varepsilon\int_{B_{r'}(x_0)}{\left\vert \nabla w\right\vert^p}
		\end{equation}		 
		for $\varepsilon<\varepsilon_0=\varepsilon_0(N,p)$.
		By \eqref{e:pharmonic_p bigger than 2} it is immediate for $p\geq 2$ (even without the second summand on the right hand side).
		For $1<p<2$ we use Young inequality { (in the form $ab\leq C_\varepsilon a^q+\varepsilon b^{q'}$ with $q=2/p$)
		to get		
		\begin{equation*}
		\begin{split}
			\int_{B_{r'}(x_0)}{\left\vert \nabla g\right\vert^p}
			=\int_{B_{r'}(x_0)}{\left\vert \nabla g\right\vert^p\left(\left\vert \nabla w\right\vert^2+\left\vert \nabla g\right\vert^2\right)^{p(p-2)/4}\left(\left\vert \nabla w\right\vert^2+\left\vert \nabla g\right\vert^2\right)^{-p(p-2)/4}}\\
			\leq
			C_\varepsilon\int_{B_{r'}(x_0)}{\left\vert \nabla g\right\vert^2\left(\left\vert \nabla w\right\vert^2+\left\vert \nabla g\right\vert^2\right)^{(p-2)/2}}+\varepsilon\int_{B_{r'}(x_0)}{\left(\left\vert \nabla w\right\vert^2+\left\vert \nabla g\right\vert^2\right)^{p/2}}\\
			\leq C_\varepsilon(r')^N+C\varepsilon\int_{B_{r'}(x_0)}{\left(\left\vert \nabla w\right\vert^p+\left\vert \nabla g\right\vert^p\right)},
		\end{split}	
		\end{equation*}}
		yielding \eqref{e:bound on dirichlet energy of g} for $\varepsilon$ small enough. Note that in the last inequality we used \eqref{e:pharmonic_p less than 2}.

\noindent
So we have
		\begin{equation*}
			\begin{aligned}
				&\int_{B_r(x_0)}{\left\vert \nabla u\right\vert^p}dx
				\leq C\left(\left(\frac{r}{r'}\right)^{N+p\beta-p}+\varepsilon\right)\int_{B_{r'}(x_0)}{\left\vert\nabla w\right\vert^p}dx+C_\varepsilon(r')^N\\
				&\qquad\leq C\left(\left(\frac{r}{r'}\right)^{N+p\beta-p}+\varepsilon\right)\int_{B_{r'}(x_0)}{\left\vert\nabla u\right\vert^p}dx+C_\varepsilon(r')^N\\
				&\qquad\qquad+C\left(\left(\frac{r}{r'}\right)^{N+p\beta-p}+\varepsilon\right)\int_{B_{r'}(x_0)}{\left\vert\nabla g\right\vert^p}dx\\
				&\qquad\leq C\left(\left(\frac{r}{r'}\right)^{N+p\beta-p}+\varepsilon\right)\int_{B_{r'}(x_0)}{\left\vert\nabla u\right\vert^p}dx+C_\varepsilon(r')^N,
			\end{aligned}	
		\end{equation*}
 		which gives us
		\begin{equation*}
			\phi(r)\leq c\left(\left(\frac{r}{r'}\right)^{N+p\beta-p}+\varepsilon\right)\phi(r')+C_\varepsilon(r')^{N}.
		\end{equation*}
		Using Lemma \ref{techlemmagrowth} we obtain
		\begin{equation*}
			\phi(r)\leq c\left(\left(\frac{r}{r'}\right)^{N+p\beta-p}\phi(r')+Cr^{N}\right)
		\end{equation*}
		for any $r<r'<1$. In particular,
		\begin{equation*}
			\phi(r)\leq c\left(\Vert u\Vert^p_{L^p(\mathbb{R}^N)}+C\right)r^{N+p\beta-p}.
		\end{equation*}	
	\end{proof}

	\subsubsection{Lipschitz continuity and density estimates on the boundary}

	The following lemma is an analogue of \cite[Lemma 3.2]{DP}
	and it will give us uniform Lipschitz continuity.
		
	\begin{lemma} \label{lip}
		There exists $M=M(N,R)$  such that if \(u_j\) is the capacitary potential of a minimizer for \eqref{pcappertpb} and \(v_j=1-u_j\) satisfies
		$v_j(x_0)=0$,
		then 
		\begin{equation*}
			\sup_{B_{r/4}(x_0)}{v_j}\leq Mr.
		\end{equation*}
	\end{lemma}
	
	\begin{proof}
		{\bf Step 1.} We argue by contradiction and get a sequence
		$v_{j_k}$, $B_{r_k}(y_k)\subset B_R$ such that $v_{j_k}(y_k)=0$, $\sup_{B_{r_k/4}(y_k)}{v_{j_k}}\geq k r_k$.
		We now consider blow-ups around $y_k$, that is, we define
		$$\tilde{v}_k(x):=\frac{v_{j_k}(y_k+r_k x)}{r_k}.$$
		Note that $\tilde{v}_k$ minimizes
		\begin{equation*}
			\int_{\mathds{R}^N}{\vert\nabla v\vert^p}dx+
			r_k^{-N}f_{\hat{\eta}}(r_k^N\{v=0\})+
			r_k^{-N}\sqrt{\varepsilon_{j_k}^2+\sigma^2(\alpha(\Phi_k(\{v=0\}))-\varepsilon_{j_k})^2}
		\end{equation*}
		among functions such that $\Phi_k(\{v=0\})\subset B_R$, where $\Phi_k(x)=y_k+r_k x$. Additionally, we have
		\begin{equation*}
			\tilde{v}_k(0)=0,\,\sup_{B_{1/4}}\tilde{v}_k\geq k.
		\end{equation*}
		We define a function
		\begin{equation*}
			d_k(x):=\mathrm{dist}\left(x, \{\tilde{v}_k=0\}\right)
		\end{equation*}
		and a set
		\begin{equation*}
			V_k:=\left\{x\in B:d_k(x)\leq\frac{1-\vert x\vert}{3}\right\}.
		\end{equation*}
		The following properties hold for $V_k$:
		\begin{itemize}
			\item $B_{1/4}\subset V_k$. This is due to the fact that
			$\tilde{v}_k(0)=0$ and thus $d_k(x)\leq\vert x\vert$.
			\item $m_k:=\sup_{x\in V_k}{(1-\vert x\vert)\tilde{v}_k(x)}\geq\frac{3k}{4}$. This follows from the previous property
			and the fact that $\sup_{B_{1/4}}\tilde{v}_k\geq k$.
		\end{itemize}
		Since $\tilde{v}_k$ is continuous and $(1-\vert x\vert)\tilde{v}_k(x)=0$ on $\partial B$, $m_k$ is obtained at some point 
		$x_k\in V_k$. We notice that the following holds for $x_k$:
		\begin{equation*}
			\tilde{v}_k(x_k)=\frac{m_k}{1-\vert x_k\vert}\geq m_k\geq\frac{3k}{4};\quad\delta_k:=d_k(x_k)\leq\frac{1-\vert x_k\vert}{3}.
		\end{equation*}
		We now take projections of $x_k$ onto $\{\tilde{v}_k=0\}$,
		that is, we consider a sequence $z_k$ such that 
		$z_k\in\{\tilde{v}_k=0\}\cap B$, $\vert z_k-x_k=\delta_k\vert$. Note that $B_{2\delta_k}(z_k)\subset B$. Moreover, $B_{\delta_k/2}(z_k)\subset V_k$ since for any $x\in B_{\delta_k/2}(z_k)$ we have
		\begin{equation*}
			1-\vert x\vert\geq 1-\vert x_k\vert-\vert x_k-x\vert
			\geq 1-\vert x_k\vert-\frac{3}{2}\delta_k\geq\frac{1-\vert x_k\vert}{2}.
		\end{equation*}
		Now let us show that $\sup_{B_{\delta_k/4}}{\tilde{v}_k}\sim\tilde{v_k}(x_k)$. Indeed, for the upper bound it is enough to notice that
		\begin{equation*}
			\sup_{B_{\delta_k/2}(z_k)}\tilde{v}_k\leq
			\tilde{v}_k(x_k)(1-\vert x_k\vert)\sup_{B_{\delta_k/2}(z_k)}{\frac{1}{1-\vert x\vert}}\leq 2\tilde{v}_k(x_k).
		\end{equation*}
		On the other hand, since $B_{\delta_k}(x_k)\subset\{\tilde{v}_k>0\}$, $\tilde{v}_k$ is $p$-harmonic in $B_{\delta_k}(x_k)$ and thus,
		by Harnack inequality (see, for example, \cite[Theorem 2.20]{plaplacenotes}), we get
		\begin{equation*}
			\sup_{B_{\delta_k/4}(z_k)}\tilde{v}_k\geq
			\inf_{B_{4\delta_k/5}(x_k)}\tilde{v}_k
			\geq c_0\sup_{B_{4\delta_k/5}(x_k)}\tilde{v}_k-C
			\geq \frac{c_0}{2}\tilde{v}_k(x_k),
		\end{equation*}
		where the last inequality holds for $k$ big enough.
		
		{\bf Step 2.} We now consider blow-ups around $z_k$. We define
		\begin{equation*}
			\hat{v}_k(x):=\frac{\tilde{v}_k(z_k+\frac{\delta_k}{2}x)}{\tilde{v}_k(x_k)}.
		\end{equation*}
		We note that
		\begin{equation*}
			\sup_B\hat{v}_k\leq 2,\quad \sup_{B_{1/2}}\hat{v}_k\geq c_0/2,\quad \hat{v}_k(0)=0
		\end{equation*}
		and $\hat{v}_k$ is a minimizer of
		\begin{equation}\label{e:min problem for the second blow-up}
		\begin{split}
			\int_{\mathds{R}^n}{\vert\nabla v\vert^p}dx+
			\left(\delta_k/2\right)^{p-n}\tilde{v}_k(x_k)^{-p}r_k^{-n}f_{\hat{\eta}}((\delta_k/2)^n r_k^n\{v=0\})+\\
			\left(\delta_k/2\right)^{p-n}\tilde{v}_k(x_k)^{-p}r_k^{-n}\sqrt{\varepsilon_{j_k}^2+\sigma^2(\alpha(\Psi_k(\{v=0\}))-\varepsilon_{j_k})^2}
		\end{split}	
		\end{equation}
		among functions such that $\Psi_k(\{v=0\})\subset B_R$,
		where $\Psi_k(x)=y_k+r_k z_k +\frac{r_k\delta_k x}{2}$.
		
		We introduce $w_k$ - a $p$-harmonic continuation of $\hat{v}_k$ in $B_{3/4}$:
		\begin{equation*}
		\begin{cases}
		\div(\nabla\vert w_k\vert^{p-2}\nabla w_k)=0\text{ in }B_{3/4},\\
		w_k=\hat{v}_k\text{ in }B_{3/4}^c.
		\end{cases}
		\end{equation*}
		By maximum principle (see, for example, \cite[Corollary 2.21]{plaplacenotes}) $w_k>0$ in $B_{3/4}$ and thus
		\begin{equation*}
		\{\hat{v}_k=0\}\Delta\{w_k=0\}=\{\hat{v}_k=0\}\cap B_{3/4}.
		\end{equation*}
		So now, remembering that $\hat{v}_k$ is a minimizer
		for \eqref{e:min problem for the second blow-up} and
		using $w_k$ as a comparison function we obtain
		\begin{equation*}
			\int_{B_{3/4}}{\vert\nabla\hat{v}_k\vert^p}\,dx
			\leq \int_{B_{3/4}}{\vert\nabla w_k\vert^p}\,dx+\frac{C}{k^p}.
		\end{equation*}
		From this we can infer the convergence of $v_k-w_k$ to zero.
		In order to do that, we define $$v_k^s=s\hat{v}_k+(1-s)w_k.$$
		Now, we write
		\begin{equation*}
		\begin{split}
			\frac{C}{k^p}\geq\int_{B_{3/4}}{\vert\nabla\hat{v}_k\vert^p}\,dx-\int_{B_{3/4}}{\vert\nabla w_k\vert^p}\,dx
			=p\int_0^1{\frac{1}{s}\,ds}\int_{B_{3/4}}{\vert\nabla v_k^s\vert^{p-2}\nabla v_k^s\cdot\nabla(v_k^s-w_k)}\,dx\\
			=p\int_0^1{\frac{1}{s}\,ds}\int_{B_{3/4}}{\left(\vert\nabla v_k^s\vert^{p-2}\nabla v_k^s-\vert\nabla w_k\vert^{p-2}\nabla w_k\right)\cdot\nabla(v_k^s-w_k)}\,dx.
		\end{split}	
		\end{equation*}
		We want to show that the convergence of $v_k-w_k$ is strong.
		We use Lemma \ref{l:inequality for L^p norm of a difference} for that. 
		We need to consider two cases. For $p\geq 2$ by 
		the inequality \eqref{e:inequality for L^p norm for p>=2} we get $$\int_{B_{3/4}}{\vert\nabla\hat{v}_k-\nabla w_k\vert^p}\,dx\leq\frac{C}{k^p},$$ yielding the strong convergence of $\hat{v}_k-w_k$ to zero in $W^{1,p}(B_{3/4})$ as $k\rightarrow\infty$.					To deal with the case $1<p<2$, we observe that $\hat{v}_k$
		is bounded in $D^{1,p}(B_{3/4})$. We infer that $w_k$ is bounded in $D^{1,p}$ too and hence, by the inequality \eqref{e:inequality for L^p norm for p<2} we also have the strong convergence of $\hat{v}_k-w_k$ to zero in $W^{1,p}(B_{3/4})$ as $k\rightarrow\infty$.
		
		We recall now that $\hat{v}_k$ is equibounded in $C^{0,\alpha}(B_{3/4})$ and hence, up to a non-relabelled subsequence we have
		that $\hat{v}_k$ converges to some continuous function $v_\infty$ locally uniformly and weakly in $W^{1,p}$. This means that
		also $w_k$ converges to $v_\infty$ weakly in $W^{1,p}$.
		Elliptic regularity for $w_k$ tell us that $w_k$ is locally bounded in $C^{1,\beta}(B_{3/4})$ and so up to a subsequence
		$w_k$ converges to $v_\infty$ strongly in $W^{1,p}$.
		But then $v_\infty$ is $p$-harmonic with $v_\infty(0)=0$,
		$\sup_{B_{1/2}}v_\infty\geq c_0/2$. This contradicts Harnack inequality.	 
	\end{proof}

	The following lemma is an analogue of \cite[Lemma 4.2]{DP}
	and the proof is almost identical.	
	
	\begin{lemma}[non-degeneracy]
		For $\kappa<1$, $\gamma>p-1$ there exists a constant $c_{nd}~=~c_{nd}(N,\kappa,\gamma,R)$ such that if \(u_j\) is a minimizer for \eqref{pcappertpb} and  \(v_j=1-u_j\) satisfies 
		\begin{equation}\label{e:small mean p-cap}
			\left(\fint_{\partial B_r(x_0)}{v_j^\gamma}\right)^\frac{1}{\gamma}\leq cr,
		\end{equation}
		then $v_j=0$ in $B_{\kappa r}(x_0)$.
	\end{lemma}
	\begin{proof}
		We will omit the subscript $j$ for convenience and write 
		$v$ instead of $v_j$. None of the bounds will depend on $j$.	
	
		First, we want to show that if \eqref{e:small mean p-cap} holds for $\partial B_r(x_0)$ and $c$ is small enough (depending
		only on $N$, $\kappa$, $\gamma$, and $R$), 
		then the inequality \eqref{e:small mean p-cap} yields
		$B_{\kappa r}\subset B_R$. The idea is that $v$ is
		sufficiently big outside of $B_R$. Indeed,
		by maximum principle
		\begin{equation}\label{e:bound on v from below p-cap}
			v(x)\geq 1-u_{B_R}(x)=1-\frac{R^\frac{n-p}{p-1}}{\vert x\vert^\frac{n-p}{p-1}}.
		\end{equation}	
		If $B_{\kappa r}(x_0)\setminus B_R\neq\emptyset$,
		then $\vert B_r(x_0)\setminus B_{R+\frac{1-\kappa}{2}r}\vert
		\geq c(\kappa)\vert B_r\vert$ and, using \eqref{e:bound on v from below p-cap}, we get
		\begin{equation*}
			\left(\fint_{\partial B_r(x_0)}{v_j^\gamma}\right)^\frac{1}{\gamma}\geq c(\kappa)\left(1-\frac{R^\frac{n-p}{p-1}}{\vert R+\frac{1-\kappa}{2}r\vert^\frac{n-p}{p-1}}\right)
			\geq c(\kappa, N, R)r,
		\end{equation*}
		contradicting \eqref{e:small mean p-cap} for $c_{nd}$ small enough.
		
		Now we define 
		\begin{equation*}
			\varepsilon:=\frac{1}{\sqrt{\kappa}r}\sup_{B_{\sqrt{\kappa}r}}v\leq C\frac{1}{r}\left(\fint_{\partial B_r(x_0)}{v^\gamma}\right)^\frac{1}{\gamma}\leq C\, c_{nd},
		\end{equation*}
		where we used Harnack inequality for
		$p$-subharmonic functions (see \cite[Theorem 1.3]{T}).
		We set $\varphi(x)=\varphi(\vert x\vert)$ to be the solution
		of 
		\begin{equation*}
			\begin{cases}
				\Delta_p \varphi=0&\text{ in }B_{\sqrt{\kappa}r}\setminus B_{\kappa r},\\
				\varphi=0&\text{ on }\partial B_{\kappa r},\\
				\varphi=1&\text{ on }\partial B_{\sqrt{\kappa} r},
			\end{cases}
		\end{equation*}	
		defined as $0$ in $B_{\kappa r}$. Now we define 
		$$v':=\varepsilon\sqrt{\kappa}r\varphi.$$
		Note that $v'\geq v$ on $\partial B_{\sqrt{\kappa}r}$.
		Finally, we define
		$$w:=\min(v,v')\text{ in }B_{\sqrt{\kappa}r},\quad w:=v\text{ in }B_{\sqrt{\kappa}r}^c,$$
		and we use $w$ as a comparison function in \eqref{pcappertpb}.			We notice that $\{w=0\}\supset\{v=0\}$ and so from minimality
		of $v$ we conclude
		\begin{equation*}
			\int_{B_{\sqrt{\kappa} r}}{\vert\nabla v\vert^p}\,dx
			+\frac{\hat{\eta}}{2}\left\vert\{w=0\}\setminus \{v=0\}\right\vert\leq
			\int_{B_{\sqrt{\kappa} r}}{\vert\nabla w\vert^p}\,dx
			=\int_{B_{\sqrt{\kappa} r}\setminus B_{\kappa r}}{\vert\nabla w\vert^p}\,dx.
		\end{equation*}						
		Now we use the definition of $w$, positivity of $v'$ in $B_{\sqrt{\kappa}r}\setminus B_{\kappa r}$, and 
		convexity of $t\rightarrow t^p$ to get
		\begin{equation*}
		\begin{split}
			\int_{B_{\kappa r}}{\vert\nabla v\vert^p}\,dx
			+\frac{\hat{\eta}}{2}\left\vert B_{\kappa r}\cap\{v>0\}\right\vert\leq
			\int_{B_{\sqrt{\kappa} r}\setminus B_{\kappa r}}{\left(\vert\nabla w\vert^p-\vert\nabla v\vert^p\right)}\,dx\\
			\leq p\int_{B_{\sqrt{\kappa} r}\setminus B_{\kappa r}}{\vert\nabla w\vert^{p-2}\nabla w\cdot\nabla (w-v)}\,dx
			=p\int_{\partial B_{\kappa r}}{\vert\nabla v'\vert^{p-2}v\nabla v'\cdot\nu}.
		\end{split}
		\end{equation*}
		From the definition of $v'$ we have
		$$\vert\nabla v'\vert\leq C\frac{\varepsilon\sqrt{\kappa}r}{\kappa r-\sqrt{\kappa}r}\leq C\varepsilon.$$
		So we have 						
		\begin{equation*}
			\int_{B_{\kappa r}}{\vert\nabla v\vert^p}\,dx
			+\frac{\hat{\eta}}{2}\left\vert B_{\kappa r}\cap\{v>0\}\right\vert\leq
			C\varepsilon^{p-1}\int_{\partial B_{\kappa r}}{v}.
		\end{equation*}
		On the other hand, by trace inequality and Young inequality, and remembering the definition of $\varepsilon$, we can get
		\begin{equation*}
		\begin{split}
			\int_{\partial B_{\kappa r}}{v}
			&\leq C\left(\frac{1}{r}\int_{B_{\kappa r}}{v}\,dx+\int_{B_{\kappa r}}{\vert\nabla v\vert}\,dx\right)\\
			&\leq C\left(\sqrt{\kappa}\varepsilon\vert B_{\kappa r}\cap\{v>0\}\vert+\frac{1}{p}\int_{B_{\kappa r}}{\vert\nabla v\vert^p}\,dx+\frac{p-1}{p}\vert B_{\kappa r}\cap\{v>0\}\vert\right)\\
			&\leq C(1+\varepsilon)\left(\int_{B_{\kappa r}}{\vert\nabla v\vert^p}\,dx+\vert B_{\kappa r}\cap\{v>0\}\vert\right).			
		\end{split}	
		\end{equation*}					
		Bringing it all together, we get
		\begin{equation*}
			\int_{B_{\kappa r}}{\vert\nabla v\vert^p}\,dx
			+\frac{\hat{\eta}}{2}\left\vert B_{\kappa r}\cap\{v>0\}\right\vert\leq
			C\varepsilon^{p-1}(1+\varepsilon)\left(\int_{B_{\kappa r}}{\vert\nabla v\vert^p}\,dx+\vert B_{\kappa r}\cap\{v>0\}\vert\right).
		\end{equation*}
		It remains to choose $c$ from the statement of the lemma
		small enough for $C\varepsilon^{p-1}(1+\varepsilon)$ to be smaller
		that $\min\{\frac{1}{2},\frac{\hat{\eta}}{4}\}$.

	\end{proof}

	As in Section 4 of \cite{DP} these two lemmas imply Lipschitz continuity of minimizers and density estimates on the boundary of minimizing domains.
	
	\begin{lemma}
		Let $v_j$ be as above, $\Omega_j=\{v_j=0\}$.
		Then $\Omega_j$ is open and there exist constants $C=C(N, R)$, $\rho_0=\rho_0(N,R)>0$ such that
		\begin{enumerate}[label=\textup{(\roman*)}]
			\item	for every $x\in B_R$
				\begin{equation*}
					\frac{1}{C}\dist(x,\Omega_j)\leq v_j\leq C \dist(x,\Omega_j);
				\end{equation*}
			\item	$v_j$ are equi-Lipschitz;
			\item	for every $x\in\partial\Omega_j$ and $r\leq\rho_0$
				\begin{equation*}
					\frac{1}{C}\leq\frac{\vert\Omega_j\cap B_r(x)\vert}{\vert B_r(x)\vert}\leq\left(1-\frac{1}{C}\right).
\end{equation*}					
		\end{enumerate} 
	\end{lemma}
	
	Applying  \cite[Theorem 5.1]{DP} to \(v_j\) (for more details on the proof see \cite[Theorem 4.5]{AC}) we also have the following lemma.
	\begin{lemma}\label{conq}
		Let $v_j$ be as above, then there exists a Borel function $q_{u_j}$ such that
		\begin{equation} \label{eqonq}
			\div(\vert\nabla v_j\vert^{p-2} \nabla v_j)=q_{v_j}\mathcal{H}^{N-1}\mres\partial^*\Omega_j.
		\end{equation}
		Moreover, $0<c\leq -q_{v_j}\leq C$, \(c=c(n,R)\), \(C=C(n,R)\) and $\mathcal{H}^{N-1}(\partial\Omega_j\backslash\partial^*\Omega_j)=0$ .
	\end{lemma}	
	Since $\Omega_j$ converge to $B_1$ in $L^1$ by Lemma \ref{pcapfirstpropmin}, the density estimates also give us
	the following convergence of boundaries.
	\begin{lemma}\label{l:conv1}
	Let \(\Omega_j\) be minimizers of \eqref{pcappertpb}. Then
 every limit point of $\Omega_j$ 
			with respect to $L^1$ convergence is the unit ball centered at some $x_\infty\in B_R$.
			Moreover, the convergence holds also in the Kuratowski sense.
	\end{lemma}
	\begin{cor} \label{closetoball}
	  In the setting of Lemma \ref{l:conv1}, for every $\delta>0$ there exists \(j_\delta\) such that for $j\ge j_\delta$	  
	  $$B_{1-\delta}(x_j)\subset\Omega_j\subset B_{1+\delta}(x_j)$$ for some $x_j\in B_R$.
	\end{cor}	
	
	\subsection{Higher regularity of the free boundary}
	
	{  To address the higher regularity of $\partial\Omega_j$, we are first going to prove that $q_{v_j}$ is smooth
	and then use the free boundary regularity result of \cite{DP}. Note that in order to apply
	\cite[Theorem 9.1]{DP} we need to know that $\partial\Omega_j$ is flat, which follows from closeness of sets $\Omega_j$
	to the ball.}
	  
	We show smoothness of $q_{v_j}$ by using the Euler-Lagrange equations for our minimizing problem. 
	We defined $\Omega_j$ in such a way that the following minimizing property holds
			\begin{equation}\label{minineq}
			\begin{aligned}
				\cpp{\mathds{R}^N}{v_j}+f_{\hat{\eta}}(\vert\{v_j=0\}\vert)+\sqrt{\varepsilon_j^2+\sigma^2(\alpha(\{v_j=0\})-\varepsilon_j)^2}\\
				\leq\cpp{\mathds{R}^N}{v}+f_{\hat{\eta}}(\vert\{v=0\}\vert)+\sqrt{\varepsilon_j^2+\sigma^2(\alpha(\{v=0\})-\varepsilon_j)^2}
			\end{aligned}	
			\end{equation}	
			for any $v\in W^{1,2}(\mathds{R}^N)$ such that $0\leq v\leq 1$, $\{v=0\}\subset B_R$.
			
		To write Euler-Lagrange equations for $v_j$, we need to have \eqref{minineq} 
	for  $v_j\circ \Phi$  where \(\Phi\) is a diffeomorphism of \(\mathds R^N\) close to the identity. Note that to make sure that \(\{v_j\circ\Phi=0\}\) is contained in \(B_R\) one needs to know that \(\dist(\{v_j=0\},\partial B_R)>0\).  This follows from Corollary \ref{closetoball}, up to translating \(\Omega_j\). More precisely we will get the following optimality condition
	
	\begin{equation*}
	(p-1)q_{v_j}^p-\frac{\sigma^2(\alpha(\Omega_j)-\varepsilon_j)}{\sqrt{\varepsilon_j^2+\sigma^2(\alpha(\Omega_j)-\varepsilon_j)^2}}\left(\vert x-x_{\Omega_j}\vert-\left(\fint_{\Omega_j}\frac{y-x_{\Omega_j}}{\vert y-x_{\Omega_j}\vert}dy\right)\cdot x\right)
	=\Lambda_j
	\end{equation*}
	for some constant $\Lambda_j>0$. This equation is an immediate consequence of the following lemma whose proof is almost the same as \cite[Lemma 4.15]{fkstab} (which in turn is based on~\cite{AguileraAltCaffarelli86}). For this reason we only highlight the most relevant changes, referring the reader to \cite[Lemma 4.15]{fkstab} for more details.
	
	\begin{lemma}\label{l:EL} There exists $j_0$ such that for any $j\geq j_0$ and any two points $x_1$ and $x_2$
	in the reduced boundary of $\Omega_j$ the following equality holds:
		\begin{equation*}
		\begin{aligned}
			&(p-1)q_{v_j}^p(x_1)-\frac{\sigma^2(\alpha(\Omega_j)-\varepsilon_j)}{\sqrt{\varepsilon_j^2+\sigma^2(\alpha(\Omega_j)-\varepsilon_j)^2}}\left(\vert x_1-x_{\Omega_j}\vert-\left(\fint_{\Omega_j}\frac{y-x_{\Omega_j}}{\vert y-x_{\Omega_j}\vert}dy\right)\cdot x_1\right)\\
			&=(p-1)q_{v_j}^p(x_2)-\frac{\sigma^2(\alpha(\Omega_j)-\varepsilon_j)}{\sqrt{\varepsilon_j^2+\sigma^2(\alpha(\Omega_j)-\varepsilon_j)^2}}\left(\vert x_2-x_{\Omega_j}\vert-\left(\fint_{\Omega_j}\frac{y-x_{\Omega_j}}{\vert y-x_{\Omega_j}\vert}dy\right)\cdot x_2\right).
		\end{aligned}	
		\end{equation*}
	\end{lemma}	
	\begin{proof}
		We argue by contradiction. Assume there exist $x_1,x_2\in\partial^*\{v_j=0\}$ 
		such that
		\begin{equation}\label{contrEL1}
		\begin{aligned}
			&(p-1)q_{v_j}^{p-1}(x_1)-\frac{\sigma^2(\alpha(\Omega_j)-\varepsilon_j)}{\sqrt{\varepsilon_j^2+\sigma^2(\alpha(\Omega_j)-\varepsilon_j)^2}}\left(\vert x_1-x_{\Omega_j}\vert-\left(\fint_{\Omega_j}\frac{y-x_{\Omega_j}}{\vert y-x_{\Omega_j}\vert}dy\right)\cdot x_1\right)\\
			&<(p-1)q_{v_j}^{p-1}(x_2)-\frac{\sigma^2(\alpha(\Omega_j)-\varepsilon_j)}{\sqrt{\varepsilon_j^2+\sigma^2(\alpha(\Omega_j)-\varepsilon_j)^2}}\left(\vert x_2-x_{\Omega_j}\vert-\left(\fint_{\Omega_j}\frac{y-x_{\Omega_j}}{\vert y-x_{\Omega_j}\vert}dy\right)\cdot x_2\right).
		\end{aligned}	
		\end{equation}
		Using this inequality, we are going to construct a variation contradicting \eqref{minineq}. We take a smooth radial symmetric function $\phi(x)=\phi(\vert x\vert)$ supported in $B_1$ and define the following diffeomorphism for small $\tau$ and $\rho$:
		\begin{equation*}
			\Phi_\tau^\rho(x)=
			\begin{cases}
				x+\tau\rho\phi(\vert\frac{x-x_1}{\rho}\vert)\nu(x_1), &x\in B_\rho(x_1),\\
				x-\tau\rho\phi(\vert\frac{x-x_2}{\rho}\vert)\nu(x_2), &x\in B_\rho(x_2),\\
				x, &\text{otherwise.}
			\end{cases}
		\end{equation*}
		We define the function $$v^\rho_\tau:=v\circ(\Phi_\tau^\rho)^{-1}$$ and we define a competitor domain $\Omega_\tau^\rho$ as follows: 
		 $$\Omega_\tau^\rho:=\{v^\rho_\tau=0\}.$$
		Note that $1-v_\tau^\rho$ is a competitor for the $p$-capacity of $\Omega_\tau^\rho$, so we have
		$$\Capa_p(\Omega_\tau^\rho)\leq\int_{(\Omega_\tau^\rho)^c}{\vert\nabla v_\tau^\rho\vert^p}.$$
		
		Now we are going to show that for $\tau$ and $\rho$ small enough $\mathscr{C}_{\hat{\eta}}(\Omega_\tau^\rho)<\mathscr{C}_{\hat{\eta}}(\Omega)$. To do that, we first compute the variation of all the terms involved in $\mathscr{C}_{\hat{\eta}}$. 

\medskip		
\noindent
		\textbf{Volume.}
		By arguing as in  \cite[Lemma 4.15]{fkstab} one gets
		\begin{equation*}
		\begin{aligned}
			\vert\Omega_\tau^\rho\vert-\vert\Omega\vert
                         &=\tau\rho^N\left(\int_{\{y\cdot\nu(x_1)=0\}\cap B_1}{\phi\left(\left\vert y\right\vert\right)}
			-\int_{\{y\cdot\nu(x_2)=0\}\cap B_1}{\phi\left(\left\vert y\right\vert\right)}\right)
			+o(\tau)\rho^N+o_\tau(\rho^N)
			\\
			&=o(\tau)\rho^N+o_\tau(\rho^N),
		\end{aligned}
		\end{equation*}
where \(o_\tau(\rho^N)\rho^{-N}\) goes to zero as \(\rho \to 0\) and  \(o(\tau)\) is independent on \(\rho\).

\medskip		
\noindent
\textbf{Barycenter.} Assume that  that $x_\Omega=0$, as in \cite[Lemma 4.15]{fkstab} one gets,

			\[			x_{\Omega_\tau^\rho}=-\rho^N\tau\frac{x_1-x_2}{\vert\Omega\vert}\left(\int_{\{y_1=0\}\cap B_1}{\phi(\vert y\vert)}\right)+\rho^No(\tau)+o_\tau(\rho^N).
		\]
		
		\medskip
		\noindent
		\textbf{Asymmetry.}
		Again by the very same computations as in \cite[Lemma 4.15]{fkstab} one gets
		\begin{equation*}
		\begin{split}
		\alpha(\Omega_\tau^\rho)-\alpha(\Omega)
			=-\rho^N\tau\left(\int_{\{y_1=0\}\cap  B_1}{\phi(\vert y\vert)}\right)\Big(\vert x_1\vert-\vert x_2\vert+\left(\fint_\Omega{\frac{y}{\vert y\vert}dy}\right)\cdot(x_1-x_2)\Big)\\
			+o(\tau)\rho^N+o_\tau(\rho^N).
		\end{split}	
		\end{equation*}	
		\medskip
		\noindent		
		\textbf{Dirichlet energy}. Here one can argue as in \cite[Lemma 3.19]{FZ} to get 
		\begin{equation*}
		\begin{split}
			\Capa_p(\Omega_\tau^\rho)-\Capa_p(\Omega)\le \tau\rho^N(p-1)\left(\vert q(x_1)\vert^p-\vert q(x_2)\vert^p\right)\int_{B_1\cap\{y_1=0\}}{\phi(\vert y\vert)}dy\\
			+o(\tau)\rho^N+o_\tau(\rho^N).
		\end{split}	
		\end{equation*}
Combining the above estimates one gets
			\begin{equation*}
			\begin{aligned}
				&\left(\int_{B_1\cap\{y_1=0\}}{\phi(\vert y\vert)}dy\right)^{-1}\frac{\mathscr{C}_{\hat{\eta},j}(\Omega_\tau^\rho)-\mathscr{C}_{\hat{\eta},j}(\Omega)}{\rho^N}
				=\tau\left((p-1)\left(\vert q(x_1)\vert^p-\vert q(x_2)\vert^p\right)\right)\\
				&-\tau\frac{\sigma^2(\alpha(\Omega)-\varepsilon_j)}{\sqrt{\varepsilon_j^2+\sigma^2(\alpha(\Omega)-\varepsilon_j)^2}}\left(\vert x_1\vert-\vert x_2\vert+\left(\fint_\Omega{\frac{y}{\vert y\vert}dy}\right)\cdot(x_1-x_2)\right)+o(\tau)+o_\tau(1).
			\end{aligned}	
			\end{equation*}
	According to \eqref{contrEL1} the quantity in  parentheses is strictly negative. Thus, we get a contradiction with the minimality of $\Omega$ for \(\rho\) and $\tau$ small enough.
	\end{proof}

	\begin{lemma}[Smoothness of $q_v$]
		There exist constants $\delta=\delta(N,R)>0$, $j_0=j_0(N,R)$, $\sigma_0=\sigma_0(N,R)>0$
		such that for every $j\geq j_0$, $\sigma\leq\sigma_0$ the functions $q_{v_j}$ belong to $C^\infty(\mathcal{N}_\delta(\partial\Omega_j))$.
		
		Moreover, for every $k$ there exists a constant $C=C(k,N,R)$ such that
		\begin{equation*}
			\Vert q_{v_j}\Vert_{C^k(\mathcal{N}_\delta(\partial\Omega_j))}\leq C
		\end{equation*}
		for every $j\geq j_0$.
	\end{lemma}
	\begin{proof}
		We would like to write an explicit formula for $q_{v_j}$ using Euler-Lagrange equations, namely
	\begin{equation*}
		q_{v_j}=-\left(\frac{\sigma^2(\alpha(\Omega_j)-\varepsilon_j)}{\sqrt{\varepsilon_j^2+\sigma^2(\alpha(\Omega_j)-\varepsilon_j)^2}}\left(\vert x-x_{\Omega_j}\vert-\left(\fint_{\Omega_j}\frac{y-x_{\Omega_j}}{\vert y-x_{\Omega_j}\vert}dy\right)\cdot x\right)
	+\Lambda_j\right)^\frac{1}{p}.
	\end{equation*}
	To do that, we need to show that the quantity in the parenthesis is bounded away from zero.
	Indeed, $q_{v_j}$ is bounded from above and below independently of $j$ and
	\begin{equation*}
		\left\vert\frac{\sigma^2(\alpha(\Omega_j)-\varepsilon_j)}{\sqrt{\varepsilon_j^2+\sigma^2(\alpha(\Omega_j)-\varepsilon_j)^2}}\left(\vert x-x_{\Omega_j}\vert-\left(\fint_{\Omega_j}\frac{y-x_{\Omega_j}}{\vert y-x_{\Omega_j}\vert}dy\right)\cdot x\right)
	\right\vert\leq C(N,R)\sigma.
	\end{equation*}
	Then it follows from the Euler-Lagrange equations that also 
	$\Lambda_j$ is bounded from above and below
	independently of $j$. Thus, for $\sigma$ small enough we can write the above-mentioned explicit formula for $q_{v_j}$ and get the conclusion of the lemma.
	\end{proof}
	
Now we want to apply the results of \cite{DP}. We can't apply them directly, since the equation there is slightly different. More precisely, in \cite{DP} the authors are considering solutions of the equation
\[	\div\left(\vert\nabla u\vert^{p-2}\nabla u\right)=\mathcal{H}^{n-1}\mres\partial\{u>0\},
\] 
whereas \(v_j\) satisfies
\[	\div\left(\vert\nabla v_j\vert^{p-2}\nabla v_j\right)=q_{v_j}\mathcal{H}^{n-1}\mres\partial\{v_j>0\}.
\] 
However, since $q_{v_j}$ is smooth, the proof works in exactly the same way	(see also Appendix of \cite{FZ}	for the same result for a slightly different equation, the proof becomes more involved in that case).
	The idea is that flatness improves in smaller balls if the free
	boundary is sufficiently flat in some ball.

First, we need to recall the definition of flatness for the free boundary, see  \cite[Definition 7.1]{AC} (here it is applied to \(v\)).

	\begin{defin} 
		Let $\mu_-,\mu_+\in(0, 1]$. A weak solution $v$ of \eqref{eqonq} is said to be of class
$F(\mu_-,\mu_+,\infty)$ in $B_\rho(x_0)$ in a direction $\nu\in S^{N-1}$ if	$x_0\in\partial\{v=0\}$ and
				\[		
				\begin{cases}
					v(x)=0 \qquad&\text{ for }(x-x_0)\cdot\nu\leq-\mu_-\rho,\\
					v(x)\geq q_v(x_0)((x-x_0)\cdot\nu-\mu_+\rho) &\text{ for } (x-x_0)\cdot\nu\geq\mu_+\rho.\\
				\end{cases}
		\]
{ 		See also Figure \ref{fig:flat}.}
	\end{defin}
	{ 
	\begin{rem}
		In particular, if the free boundary is sufficiently flat near $x_0$, that is, it is contained in a thin strip,
		then the corresponding $v$ is of class $F(\mu,1,\infty)$ in a small ball around $x_0$ for some $\mu$.
	\end{rem}}
	
\begin{figure}[t]
\includegraphics[width=5in]{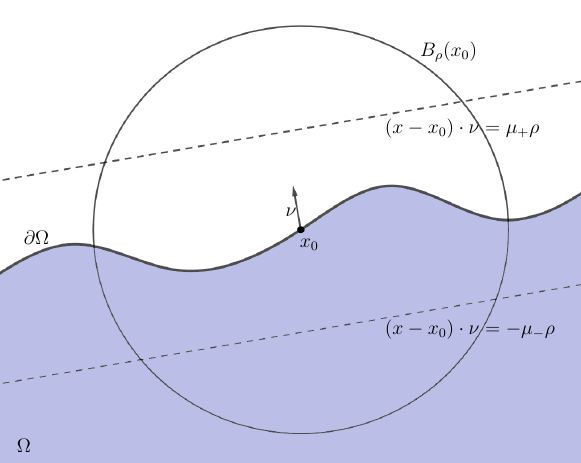}
\caption{$v$ of the class $F(\mu_-,\mu_+,\infty)$ in $B_\rho(x_0)$ in a direction $\nu\in S^{N-1}$}
\centering
\label{fig:flat}
\end{figure}	
	
	
	We are going to use that flat free boundaries are  smooth.
	The following theorem is a slight generalization of
	\cite[Theorem 9.1]{DP} and we omit the proof since it is almost identical.
	
	\begin{theorem}\label{thm:pcapregflat}  Let $u$ be a weak solution of \eqref{eqonq} and assume that $q_v$ is Lipschitz continuous. There are constants $\gamma,\mu_0,\kappa, C$ such that if $v$ is of class $F(\mu, 1, \infty)$ in $B_{4\rho}(x_0)$ in some direction $\nu\in S^{N-1}$ with $\mu\leq\mu_0$ and $\rho\leq\kappa\mu^2$, then there exists a $C^{1,\gamma}$ function $f:\mathds{R}^{N-1}\rightarrow\mathds{R}$ with $\Vert f\Vert_{C^{1,\gamma}}\leq C\mu$ such that
\begin{equation}
	\partial\{v=0\}\cap B_\rho(x_0)=(x_0+\graph_\nu f)\cap B_\rho(x_0),
\end{equation}
where $\graph_\nu f=\{x\in\mathds{R}^N:x\cdot\nu=f(x-x\cdot\nu)\nu)\}$.

Moreover if $q_v\in C^{k,\gamma}$ in some neighborhood of $\{u_j=1\}$, then $f\in C^{k+1,\gamma}$ and 
$\Vert f\Vert_{C^{k+1,\gamma}}\leq C(N,R,\Vert q_v\Vert_{C^{k,\gamma}})$.
	\end{theorem}
	
	\begin{proof}[Proof of Theorem \ref{pcapSelPr}]
		We define $\Omega_j$ as minimizers of \eqref{pcappertpb}. To get the desired sequence we
		will rescale the sets $\{\Omega_j\}$ so that they have the correct volume.
		We need to show that $\{\Omega_j\}$ converges smoothly to the ball $B_1$. Indeed one then define 
		\[
		U_j=\lambda_j(\Omega_j-x_{\Omega_j}).
		\]
		Lemma \ref{pcapfirstpropmin} then implies all the desired properties of \(U_j\), compare with  \cite[Proof of Proposition 4.4]{fkstab}.
		
		Let $\mu_0$, $\kappa$ be as in Theorem \ref{thm:pcapregflat} and $\mu<\mu_0$ to be fixed later. Let $\overline{x}$ be some point on the boundary of $B_1$. As $\partial B_1$ is smooth, it lies inside a narrow strip in the neighborhood of $\overline{x}$. More precisely, there exists $\rho_0=\rho_0(\mu)\leq\kappa\mu^2$ such that for every $\rho<\rho_0$ and every $\overline{x}\in\partial B_1$
		$$\partial B_1\cap B_{5\rho}(\overline{x})\subset\{x: \vert(x-\overline{x})\cdot\nu_{\overline{x}}\vert\leq\mu\rho\}.$$
		
		We know that $\partial\Omega_j$ are converging to $\partial B_1$ in the sense of Kuratowski. Thus, 
		there exists a point $x_0\in\partial\Omega_j\cap B_{\mu\rho_0}(\overline{x})$ such that
		$$\partial \Omega_j\cap B_{4\rho_0}(x_0)\subset\{x: \vert(x-x_0)\cdot\nu_{\overline{x}}\vert\leq 4\mu\rho_0\}.$$
		So, $u_j$ is of class $F(\mu,1,\infty)$ in $B_{4\rho_0}(x_0)$ with respect to the direction $\nu_{\overline{x}}$ and by Theorem \ref{thm:pcapregflat}, $\partial \Omega_j\cap B_{\rho_0}(x_0)$ is the graph of a smooth function with respect to $\nu_{\overline{x}}$. More precisely, for $\mu$ small enough there exists a family of smooth functions $g^{\overline{x}}_j$ with uniformly bounded $C^k$ norms such that
		$$\partial \Omega_j\cap B_{\rho_0}(\overline{x})=\{x+g^{\overline{x}}_j(x)x:x\in\partial B_1\}\cap B_{\rho_0}(\overline{x}).$$
		By a covering argument this gives a family of smooth functions $g_j$ with uniformly bounded $C^k$ norms such that
		$$\partial \Omega_j=\{x+g_j(x)x:x\in\partial B_1\}.$$
		By Ascoli-Arzel\`a and convergence to $\partial B_1$ in the sense of Kuratowski, we get that $g_j\rightarrow 0$ in $C^{k-1}(\partial B_1)$, hence the smooth convergence of $\partial\Omega_j$.
	\end{proof}

	\section{Reduction to  bounded sets}\label{redtobdd}
	To complete the proof of Theorem \ref{thm:main p-cap} one needs to show that one can consider only sets with uniformly bounded diameter. To this end let us introduce the following.  
		\begin{defin}
		Let $\Omega$ be an open set in $\mathds{R}^n$ with $\vert\Omega\vert=\vert B_1\vert$.
		Then we define the deficit of $\Omega$ as the difference between its $p$-capacity and the $p$-capacity of the unit ball:
		$$D(\Omega)=\Capa_p(\Omega)-\Capa_p(B_1).$$
	\end{defin}	
	
	Here is the key lemma for reducing Theorem \ref{thm:main p-cap} to Theorem \ref{pcapmainthmbdd}.	
	
	\begin{lemma}\label{reducetobdd}
		There exist constants $C=C(N)$, $\delta=\delta(N)>0$ and $d=d(N)$ such that for any $\Omega\subset\mathds{R}^n$  open with $\vert\Omega\vert=\vert B_1\vert$ and $D(\Omega)\leq\delta$,
		we can find a new set $\tilde{\Omega}$ enjoying the following properties
		\begin{enumerate}
			\item $\diam(\tilde{\Omega})\leq d$,
			\item $\vert\tilde{\Omega}\vert=\vert B_1\vert$,
			\item $D(\tilde{\Omega})\leq CD(\Omega)$,
			\item \label{boundonasym} $\mathcal{A}(\tilde{\Omega})\geq\mathcal{A}(\Omega)-CD(\Omega)$.
		\end{enumerate}
	\end{lemma}

We are going to define  $\tilde{\Omega}$  as a  suitable dilation  of $\Omega\cap B_S$ for some large \(S\). Hence, we first show the following estimates on the $p$-capacity of $\Omega\cap B_S$.
 	
 	\begin{lemma} \label{estcap}
 		Let $S'>S$. Then there exists a constant $c=c(S')$ such that for any open set 	
 		$\Omega\subset\mathds{R}^N$ with $\vert\Omega\vert=\vert B_1\vert$ the following inequalities hold:
 		\begin{equation*}
 			\Capa_p(B_1)\left(1-\frac{\vert\Omega\setminus B_S\vert}{\vert B_1\vert}\right)^\frac{N-p}{N}\leq \Capa_p(\Omega\cap B_S)\leq \Capa(\Omega)-c\left(1-\frac{S}{S'}\right)^p\vert\Omega\setminus B_{S'}\vert^\frac{N-p}{N}.
 		\end{equation*}
 	\end{lemma}

 	\begin{proof}[{Proof of Lemma \ref{estcap}}]
 		The first inequality is a direct consequence of the classical isocapacitary inequality. To prove the second one we are going to use the estimates for the capacitary potential of $B_S$ for which the exact formula can be written. Denote by $u_{\Omega}$ and $u_S$ the capacitary potentials of $\Omega$ and $\Omega\cap B_S$ respectively. We first write 		
 		\begin{equation*}
			\Capa_p(\Omega\cap B_S)=\Capa_p(\Omega)+\int_{\mathds{R}^n}{\vert\nabla u_S\vert^p-\vert\nabla u_\Omega\vert^p}
			=\Capa_p(\Omega)-\int_{(\Omega\cap B_S)^c}{\left(\vert\nabla u_\Omega\vert^p-\vert\nabla u_S\vert^p\right)}.
		\end{equation*}	
		Let us show that
		\begin{equation}\label{e:bound from below for difference of capacities}
		\int_{(\Omega\cap B_S)^c}{\vert\nabla u_\Omega\vert^p-\vert\nabla u_S\vert^p}\geq c(p)\int_{\Omega\backslash B_S}\vert\nabla u_S\vert^p.		
		\end{equation}
		We first claim that
		\begin{equation}\label{e:first derivative of p power weak}
		\int_{(\Omega\cap B_S)^c}\vert \nabla u_S\vert^{p-2}\nabla u_S\cdot\nabla (u_\Omega-u_S)\geq 0.
		\end{equation}
		Note that \eqref{e:first derivative of p power weak} becomes equality for a smooth domain $\Omega$. We need to be more careful for an arbitrary domain however. Since $\Omega\cap B_S\subset\Omega$, for any $\varepsilon\geq 0$ the function $u_S^\varepsilon:=(1-\varepsilon)u_S+\varepsilon u_\Omega$ is a competitor for the $p$-capacity of $\Omega\cap B_S$. Thus,
		\begin{equation*}
		\begin{split}
			0\geq\int_{(\Omega\cap B_S)^c}\vert \nabla u_S\vert^{p}-\int_{(\Omega\cap B_S)^c}\vert \nabla u^\varepsilon_S\vert^{p}\geq p\int_{(\Omega\cap B_S)^c}\vert \nabla u^\varepsilon_S\vert^{p-2}\nabla u^\varepsilon_S\cdot\nabla (u_S-u^\varepsilon_S)\\
			=p\varepsilon\int_{(\Omega\cap B_S)^c}\vert (1-\varepsilon)\nabla u_S+\varepsilon\nabla u_\Omega\vert^{p-2}\left( (1-\varepsilon)\nabla u_S+\varepsilon\nabla u_\Omega\right)\cdot\nabla (u_S-u_\Omega).
		\end{split}			 		
		\end{equation*}
		We consider separately two cases. If $1<p<2$, by concavity of $t\rightarrow t^{p-2}$ we have		
		\begin{equation*}
		\begin{split}
			0&\geq\int_{(\Omega\cap B_S)^c}\vert (1-\varepsilon)\nabla u_S+\varepsilon\nabla u_\Omega\vert^{p-2}\left( (1-\varepsilon)\nabla u_S+\varepsilon\nabla u_\Omega\right)\cdot\nabla (u_S-u_\Omega)\\
			&\geq\int_{(\Omega\cap B_S)^c}\left((1-\varepsilon)\vert \nabla u_S\vert^{p-2}+\varepsilon\vert \nabla u_\Omega\vert^{p-2}\right)\left( (1-\varepsilon)\nabla u_S+\varepsilon\nabla u_\Omega\right)\cdot\nabla (u_S-u_\Omega)\\
			&=(1-\varepsilon)^2\int_{(\Omega\cap B_S)^c}\vert \nabla u_S\vert^{p-2}\nabla u_S\cdot\nabla (u_S-u_\Omega)\\
			&+\varepsilon(1-\varepsilon)\int_{(\Omega\cap B_S)^c}\vert \nabla u_\Omega\vert^{p-2}\nabla u_S\cdot\nabla (u_S-u_\Omega)
			+\varepsilon^2\int_{(\Omega\cap B_S)^c}\vert \nabla u_\Omega\vert^{p-2}\nabla u_\Omega\cdot\nabla (u_S-u_\Omega)\\
			&+\varepsilon^2\int_{(\Omega\cap B_S)^c}\vert \nabla u_S\vert^{p-2}\nabla u_\Omega\cdot\nabla (u_S-u_\Omega).
		\end{split}			 		
		\end{equation*}
		We know that the integrals in the last equation are finite since both $u_\Omega$ and $u_S$ are in $D^{1,p}$. Thus,
		sending $\varepsilon$ to $0$ we get the inequality \eqref{e:first derivative of p power weak} in the case
		$1<p<2$. As for the case $p\geq 2$, we use that $\vert a+b\vert^{p-2}\geq c(p)\vert a\vert^{p-2}-C(p)\vert b\vert^{p-2}$ for some positive constants $c$ and $C$ and hence		
		\begin{equation*}
		\begin{split}
			0&\geq\int_{(\Omega\cap B_S)^c}\vert (1-\varepsilon)\nabla u_S+\varepsilon\nabla u_\Omega\vert^{p-2}\left( (1-\varepsilon)\nabla u_S+\varepsilon\nabla u_\Omega\right)\cdot\nabla (u_S-u_\Omega)\\
			&\geq\int_{(\Omega\cap B_S)^c}\left(c(p)(1-\varepsilon)^{p-2}\vert \nabla u_S\vert^{p-2}-C(p)\varepsilon^{p-2}\vert \nabla u_\Omega\vert^{p-2}\right)\left( (1-\varepsilon)\nabla u_S+\varepsilon\nabla u_\Omega\right)\cdot\nabla (u_S-u_\Omega).
		\end{split}			 		
		\end{equation*}
		As for the previous case, we can now send $\varepsilon$ to $0$ and get the inequality \eqref{e:first derivative of p power weak} for $p\geq 2$.
		
		\medskip
		\noindent
		Now we are ready to prove \eqref{e:bound from below for difference of capacities}. We consider two cases: $p\geq 2$ and $1<p<2$. For both we will be using an inequality of Lemma \ref{l:ineq for difference of powers}.
		For $p\geq 2$ we have
		 \begin{equation*}
 		\begin{split}
		&\int_{(\Omega\cap B_S)^c}{\vert\nabla u_\Omega\vert^p-\vert\nabla u_S\vert^p}\geq 
		c(p)\int_{(\Omega\cap B_S)^c}\vert\nabla(u_\Omega-u_S)\vert^p\\
		&\qquad+p\int_{(\Omega\cap B_S)^c}\vert \nabla u_S\vert^{p-2}\nabla u_S\cdot\nabla (u_\Omega-u_S)\\
		&\qquad\geq c(p)\int_{(\Omega\cap B_S)^c}\vert\nabla(u_\Omega-u_S)\vert^p
		\geq c(p)\int_{\Omega\backslash B_S}\vert\nabla(u_\Omega-u_S)\vert^p\\
		&\qquad=c(p)\int_{\Omega\backslash B_S}\vert\nabla u_S\vert^p,		
		\end{split}	
		\end{equation*}	
		where for the second inequality we used \eqref{e:first derivative of p power weak}.
		As for the case $1<p<2$, we have
 		\begin{equation*}
 		\begin{split}
		\int_{(\Omega\cap B_S)^c}{\vert\nabla u_\Omega\vert^p-\vert\nabla u_S\vert^p}\geq 
		c(p)\int_{(\Omega\cap B_S)^c}\left(\vert\nabla u_\Omega\vert^2+\vert\nabla u_S\vert^2\right)^{\frac{p-2}{2}}\vert\nabla(u_\Omega-u_S)\vert^2\\
		+p\int_{(\Omega\cap B_S)^c}\vert \nabla u_S\vert^{p-2}\nabla u_S\cdot\nabla (u_\Omega-u_S)\\
			\geq c(p)\int_{(\Omega\cap B_S)^c}\left(\vert\nabla u_\Omega\vert^2+\vert\nabla u_S\vert^2\right)^{\frac{p-2}{2}}\vert\nabla(u_\Omega-u_S)\vert^2\\
			\geq c(p)\int_{\Omega\backslash B_S}\left(\vert\nabla u_\Omega\vert^2+\vert\nabla u_S\vert^2\right)^{\frac{p-2}{2}}\vert\nabla(u_\Omega-u_S)\vert^2\\
		=c(p)\int_{\Omega\backslash B_S}\vert\nabla u_S\vert^p,		
		\end{split}	
		\end{equation*}	
		where for the last equality we used that $u_\Omega\equiv 1$ in $\Omega$.

		\medskip
		\noindent
		We would like to show that $\int_{\Omega\backslash B_S}\vert\nabla u_S\vert^p$ cannot be too small. To this end let us set  $v_S=1-u_S$. By  Sobolev's embedding  we get
	\begin{equation*}
		\int_{\Omega\backslash B_S}\vert\nabla u_S\vert^p=\int_{\Omega\backslash B_S}\vert\nabla v_S\vert^p
		\geq c(N)\left(\int_{\Omega\setminus B_S}\vert v_S\vert^{p^*}\right)^\frac{p}{p^*}, 
	\end{equation*}
	where \(p^*\) is the Sobolev exponent. Let us denote
	by $z_S$ the capacitary potential of $B_S$: 
	\[
	z_S=\Biggl(1-\frac{S^{\frac{n-p}{p-1}}}{|x|^{\frac{n-p}{p-1}}}\Biggr)_+.
	\]
	By the maximum principle, $v_S\geq z_S$, hence
	\begin{equation*}
	\begin{split}
		\int_{\Omega\setminus B_S}\vert v_S\vert^{p^*}&\geq \int_{\Omega\backslash B_S}\vert z_{S}\vert^{p^*}\\
		&\ge \int_{\Omega\backslash B_{S'}}\vert z_{S}\vert^{p^*} \geq \left(1-\left(\frac{S}{S'}\right)^{\frac{n-p}{p-1}}\right)^{p^*}\vert\Omega\setminus B_{S'}\vert.
		\end{split}
	\end{equation*}	
Hence
	\begin{equation*}
		\begin{aligned}
		\Capa_p(\Omega\cap B_S)&\leq \Capa_p(\Omega)-c(N)\left(1-\left(\frac{S}{S'}\right)^{\frac{n-p}{p-1}}\right)^p\vert\Omega\setminus B_{S'}\vert^\frac{N-p}{N}\\
		&\leq \Capa_p(\Omega)-c\left(1-\frac{S}{S'}\right)^p\vert\Omega\setminus B_{S'}\vert^\frac{N-p}{N},
		\end{aligned}	
	\end{equation*}
	concluding the proof.
 	\end{proof}
	
 We can now prove Lemma \ref{reducetobdd}. 
	\begin{proof} [Proof of Lemma \ref{reducetobdd}]
	The proof is almost identical to the proof of \cite[Lemma 6.2]{DPMM}. We repeat it here for convenience of the reader.
	
	Let us  assume without loss of generality that the ball achieving the asymmetry of $\Omega$ is $B_1$.  As was already mentioned, we are going to show that there exists an $\tilde{\Omega}$ of the form $\lambda (\Omega\cap B_S)$  for  suitable \(S\) and \(\lambda\) satisfying all the desired properties. Let us set  
	\[
	b_k:=\frac{\vert\Omega\backslash B_{2-2^k}\vert}{\vert B_1\vert}\le 1.
	\]
	Note that by Theorem \ref{t:fmp} we can assume that \(b_1\le 2\mathcal A(\Omega)\) is as small as we wish (independently on \(\Omega\) up to choose \(\delta\) sufficiently small. Lemma \ref{estcap} gives 
	\begin{equation*}
		\begin{aligned}
		\Capa_p(\Omega)&-c\left(\frac{2^{-(k+1)}}{2-2^{-(k+1)}}\right)^p b_{k+1}^\frac{N-2}{N}\geq \Capa_p(B_1)(1-b_k)^\frac{N-p}{N}\geq \Capa_p(B_1)-\Capa_p(B_1)b_k,
		\end{aligned}
	\end{equation*}
	which implies
	\begin{equation} \label{estonbk}
		c b_{k+1}\leq \left(4^{N/(N-p)}\right)^k (D(\Omega)+C b_k)^\frac{N}{N-p}.
	\end{equation}
We now claim that there exists \(\bar k\) such that 
\[
b_{\bar k}\le D(\Omega).
\]
Indeed, otherwise   by \eqref{estonbk} we would get
	\begin{equation*}
		b_{k+1}\leq C \left(4^{N/(N-p)}\right)^k (D(\Omega)+C b_k)^\frac{N}{N-p}\leq \left(4^{N/(N-p)}\right)^k C' b_k^\frac{N}{N-p}
		\leq M^k b_k^\frac{N}{N-p}
	\end{equation*}
	for all \(k\in \mathds N\), where   $M=M(N)$. Iterating the last inequality, we obtain
	\begin{equation*}
		b_{k+1}\leq (M b_1)^{(\frac{N}{N-p})^k}\xrightarrow[k\to\infty]{} 0
	\end{equation*}
	if $b_1$ is small enough, which by Theorem \ref{t:fmp}  we can assume up to choose \(\delta=\delta(N)\ll1\).

	We  define $\tilde{\Omega}$ as a properly rescaled intersection of $\Omega$ with a ball. Let $\bar k $ be such that $b_{\bar k} \le D(\Omega)$
	\begin{equation*}
		\tilde{\Omega}:=\left(\frac{\vert B_1\vert}{\vert\Omega\cap  B_R\vert}\right)^\frac{1}{N}(\Omega\cap  B_R)=(1-b_{\bar k})^{-\frac{1}{N}}(\Omega\cap  B_S),
	\end{equation*}
	where $S:=2-2^{-\bar k}\le 2$. Note that \(|\tilde \Omega|=|B_1|\). We now check all the remaining properties:
	\begin{itemize}
	\item[-] \emph{Bound on the diameter}:
	\begin{equation*}
		\diam(\tilde{\Omega})\leq 2\cdot 2 (1-D(\Omega))^{-\frac{1}{N}}\leq 4(1-\delta)^{-\frac{1}{N}}\le 4.
	\end{equation*}
	up to choose \(\delta=\delta(N)\ll1\).
	\item[-]  \emph {Bound on the deficit}:
	\begin{equation*}
	\begin{aligned}
		D(\tilde{\Omega})&=\Capa_p(\tilde{\Omega})-\Capa_p(B_1)=\Capa_p(\Omega\cap B_S)(1-b_{\bar K})^{-\frac{N-p}{N}}-\Capa_p(B_1)\\
		&\leq \Capa_p(\Omega)(1-b_{\bar k})^{-\frac{N-p}{N}}-\Capa_p(B_1)\\
		&\le  \Capa_p(\Omega)-\Capa_p(B_1)+\frac{2(N-p)\Capa_p(\Omega)}{N}b_{\bar k}\le C(N)D(\Omega).
	\end{aligned}	
	\end{equation*}
	since \(b_{\bar k} \le D(\Omega)\ll 1\) and, in particular, \(\Capa_p (\Omega)\le 2 \Capa_p(B_1)\).
	\item[-]\emph{Bound on  the asymmetry}: Let $r:=(1-b_{\bar k})^{-1} \in (1,2)$, that is $r$ is such that $\tilde{\Omega}=r^N(\Omega \cap B_S)$ with \(S=2-2^{-\bar k}\le 2\). Let $x_0$ be such that $B_1(x_0)$ is a minimizing ball for $\mathcal{A}(\tilde{\Omega})$.  Then,  recalling that \(b_{\bar k}=|B_1|^{-1}|\Omega\setminus B_S|\le C(N) D(\Omega)\),
	\begin{equation*}
	\begin{aligned}
		|B_1|\mathcal{A}(\Omega)&\leq\vert\Omega\Delta B_1\left(\frac{x_0}{r}\right)\vert\leq\vert\Omega\setminus B_S\vert+\left\vert(\Omega\cap B_S)\Delta B_1\left(\frac{x_0}{r}\right)\right\vert\\
		&\leq C D(\Omega) +\left\vert(\Omega\cap B_S)\Delta B_\frac{1}{r}\left(\frac{x_0}{r}\right)\right\vert\\
		&\quad+\left\vert B_\frac{1}{r}\left(\frac{x_0}{r}\right)\Delta B_1\left(\frac{x_0}{r}\right)\right\vert\\
		&\le CD(\Omega)+\frac{|B_1|}{r^N}\mathcal{A}(\tilde{\Omega})+\vert B_1\vert\left(1-\frac{1}{r^N}\right)\\
		&\leq CD(\Omega)+|B_1|\mathcal{A}(\tilde{\Omega})+C(N)  b_{\bar k}\\
		&\leq CD(\Omega)+|B_1|\mathcal{A}(\tilde{\Omega}).
	\end{aligned}
	\end{equation*}
	
	\end{itemize}
	\end{proof} 
	
	\section{Proof of Theorem \ref{thm:main p-cap}}\label{s:proof}
 In order to reduce it to Theorem \ref{pcapmainthmbdd}, we need to start with a set which is already close to a ball. Thanks to Theorem  \ref{t:fmp},  this can be achieved by assuming the deficit sufficiently small (the quantitative inequality being trivial in the other regime).
 		
We have now all the ingredients  to prove  Theorem \ref{thm:main p-cap}.	
	
 	\begin{proof}[Proof of Theorem \ref{thm:main p-cap}]
 	First note that if \(D(\Omega)\ge \delta_0\) then, since \(\mathcal A(\Omega)\ge 2\),
		\[
		D(\Omega)\ge 4\frac{\delta_0}{4}\ge \frac{\delta_0}{4}\mathcal A(\Omega)^2.
		\]
		Hence we can assume that \(D(\Omega)\) is as small as we wish as long as the smallness depends only on \(N\). We now take \(\delta_0\) smaller than the constant \(\delta\) in Lemma \ref{reducetobdd} and, assuming that \(D(\Omega)\le \delta_0\), we use Lemma \ref{reducetobdd} to find a set  $\tilde{\Omega}$ with \(\diam(\tilde \Omega)\le d=d(N)\) and satisfying all the properties there. In particular, up to a translation we can assume that \(\tilde \Omega\subset B_d\). Up to choosing  \(\delta_0\) smaller we can apply  Theorem \ref{t:fmp} and Lemma \ref{propasym} \ref{asymlip} to ensure that \(\alpha(\tilde \Omega)\le \varepsilon_0\) where \(\varepsilon_0=\varepsilon_0(N,d)=\varepsilon_0(N)$ is the constant appearing in the statement of  Theorem \ref{pcapmainthmbdd}. This, together with Lemma \ref{propasym}, \ref{compasym},  grants that 
		\[
		D(\tilde \Omega)\ge c(N) \alpha(\tilde \Omega)\ge c(N)\mathcal A(\tilde \Omega)^2.
		\]
		Hence, by Lemma \ref{reducetobdd} and assuming that \(\mathcal A(\Omega)\ge C D(\Omega)\) (since otherwise there is nothing to prove),
		\[
		D(\Omega)\ge cD(\tilde \Omega)\ge c A(\tilde \Omega)^2\ge c \mathcal A(\tilde \Omega)^2\ge c \mathcal A(\Omega)^2-CD(\Omega)^2
		\]
		from which the conclusion easily follows since \(D(\Omega)\le \delta_0\ll 1\).
 	\end{proof}

\appendix
\section{}	

		Here we put some inequalities that are used throughout the paper.
		
		\begin{lemma}[{\cite[Lemma 2.3]{FZ}}]\label{l:inequality for L^p norm of a difference}
		Let $p>1$. There exists $c(p)\geq 0$ such that if $\kappa\geq 0$ and $\xi,\eta\in\mathds{R}^n$ then
		\begin{equation*}
		\left(\left(\kappa^2+\vert\xi\vert^2\right)^{\frac{p-2}{2}}\xi-\left(\kappa^2+\vert\eta\vert^2\right)^{\frac{p-2}{2}}\eta\right)
		\cdot\left(\xi-\eta\right)
		\geq c\left(\kappa^2+\vert\xi\vert^2+\vert\eta\vert^2\right)^{\frac{p-2}{2}}\vert\xi-\eta\vert^2.
		\end{equation*}
		Moreover, there exists another constant $C(p)\geq 0$ such that
		if $\Omega\subset\mathds{R}^n$ is an open set and for 
		$u,\, v\in W^{1,p}(\Omega)$ and $0\leq s\leq 1$, we set
		$u^s(x)=su(x)+(1-s)v(x)$,
		then the following two inequalities hold:
		\begin{itemize}
		\item for $p\geq 2$
		\begin{equation}\label{e:inequality for L^p norm for p>=2}
			\int_{\Omega}{\vert\nabla u-\nabla v\vert^p}
			\leq C\int_0^1\frac{1}{s}\,ds{\int_{\Omega}\left(\vert\nabla u^s\vert^{p-2}\nabla u^s-\vert\nabla v\vert^{p-2}\nabla v\right)\cdot\nabla(u^s-v)};
		\end{equation}
		\item for $1<p<2$
		\begin{equation}\label{e:inequality for L^p norm for p<2}
		\begin{split}
			&\int_{\Omega}{\vert\nabla u-\nabla v\vert^p}\\
			&\leq C\left(\int_0^1\frac{1}{s}\,ds{\int_{\Omega}\left(\vert\nabla u^s\vert^{p-2}\nabla u^s-\vert\nabla v\vert^{p-2}\nabla v\right)\cdot\nabla(u^s-v)}\right)^\frac{p}{2}\left(\int_{\Omega}{\left(\vert\nabla u\vert+\vert\nabla v\vert\right)^p}\right)^{1-\frac{p}{2}}.
		\end{split}
		\end{equation}
		\end{itemize}
		\end{lemma}	

 	\begin{lemma}\label{l:ineq for difference of powers}
 		Let $x,y\in\mathds{R}^N$, $p\in(1,\infty)$. Then the following
 		inequalities hold:
 		\begin{itemize}
 			\item if $p\geq 2$, then
 			\begin{equation*}
 				\vert y\vert^p\geq \vert x\vert^p+p\vert x\vert^{p-2}x\cdot(y-x)+c\vert y-x\vert^p 
 			\end{equation*}
 			for some $c=c(p)>0$;
 			\item if $1<p<2$, then
 			\begin{equation*}
 				\vert y\vert^p\geq \vert x\vert^p+p\vert x\vert^{p-2}x\cdot(y-x)+c\vert y-x\vert^2\left(\vert x\vert^2+\vert y-x\vert^2\right)^{\frac{p-2}{2}} 
 			\end{equation*}
 			for some $c=c(p)>0$.
 		\end{itemize}
 	\end{lemma}
	\begin{proof}
		Consider a function $f:\mathds{R}^N\rightarrow \mathds{R}$
		defined as $f(x)=\vert x\vert^p$. Writing Taylor expansion for 
		$f$ we get
 		\begin{equation*}
 			\vert y\vert^p=\vert x\vert^p+p\vert x\vert^{p-2}x\cdot(y-x)+\int_0^1{(1-t)D^2f(x+t(y-x))(y-x)\cdot(y-x)}\,dt. 
 		\end{equation*}
		If $p=2$, we thus have
		\begin{equation*}
 			\vert y\vert^2=\vert x\vert^2+2x\cdot(y-x)+\frac{1}{2}\vert y-x\vert^2, 
		\end{equation*}
		which gives us a desired inequality. We shall consider
		$p\neq 2$ from now on.
		
		For $p\neq 2$ the Hessian $D^2f(x)$ looks as follows:
		\begin{equation*}
			D^2f(x)=p\vert x\vert^{p-2}Id+p(p-2)\vert x\vert^{p-4}A,
		\end{equation*}
		where $A_{i,j}=x_i x_j$.
		We notice that $$0\leq A\xi\cdot\xi\leq\vert x\vert^2\vert\xi\vert^2\text{ for any vector }\xi\in\mathds{R}^N,$$
		yielding
		\begin{equation*}
			D^2f(x)\xi\cdot\xi\geq c\vert x\vert^{p-2}\vert\xi\vert^2
			\text{ for any vector }\xi\in\mathds{R}^N,
		\end{equation*}
		where $c=c(p)>0$ ($c=p$ for $p>2$, $c=p(p-1)$ for $1<p<2$).
		
		So, we have
		\begin{equation*}
			\vert y\vert^p\geq\vert x\vert^p+p\vert x\vert^{p-2}x\cdot(y-x)
		+\vert y-x\vert^2\int_0^1{(1-t)\vert x+t(y-x)\vert^{p-2}}\,dt.
		\end{equation*}
		
		Let us consider the cases of different $p$ separately.
		First, we deal with $1<p<2$. In this case $p-2<0$ and so
		\begin{equation*}
			\int_0^1{(1-t)\vert x+t(y-x)\vert^{p-2}}\,dt
			\geq \frac{1}{4}\int_{1/4}^{3/4}{\left(\vert x\vert+t\vert y-x\vert\right)^{p-2}}\,dt
			\geq c \left(\vert x\vert^2+\vert y-x\vert^2\right)^{\frac{p-2}{2}},
		\end{equation*}
		finishing the proof of lemma in this case.
		
		To tackle the case $p>2$, we further consider two cases. If $\vert y-x\vert<2\vert x\vert$, then
		\begin{equation*}
			\int_0^1{(1-t)\vert x+t(y-x)\vert^{p-2}}\,dt
			\geq c\int_{0}^{1/4}\vert x\vert^{p-2}\,dt
			\geq c \vert y-x\vert^{p-2}.
		\end{equation*}
		If instead $\vert y-x\vert\geq 2\vert x\vert$, then
		\begin{equation*}
			\int_0^1{(1-t)\vert x+t(y-x)\vert^{p-2}}\,dt
			\geq c\int_{4/7}^{6/7}\vert y-x\vert^{p-2}\,dt
			\geq c \vert y-x\vert^{p-2}.
		\end{equation*}
		
	\end{proof}


\end{document}